\documentclass[12pt,a4paper,tikz]{amsart}
\usepackage{amsmath}
\theoremstyle{plain}

\usepackage{extarrows}
\usepackage{enumerate,amssymb,amsthm,amscd}
\usepackage{fullpage}

\usepackage{etex}

\usepackage[shortlabels]{enumitem}
\usepackage{setspace}
\usepackage{tabularx,multirow}

\usepackage[bbgreekl]{mathbbol}

\usepackage[normalem]{ulem}

\usepackage[arrow, matrix, curve]{xy}
\usepackage{MnSymbol}

\advance\hoffset-5mm \advance\textwidth40mm
\input{diagrams.tex}

\def\bdi{\begin{diagram}}
\def\edi{\end{diagram}}

\usepackage[english]{babel}
\usepackage[utf8]{inputenc}
\usepackage{amsmath, amsthm, amssymb}
\usepackage[pdftex]{graphicx,color}
\usepackage{caption}
\usepackage{subcaption}
\usepackage[dvipsnames]{xcolor} 
\usepackage{tikz,pgf}
\usetikzlibrary{calc}
\usepackage[colorlinks=true,linkcolor=RoyalBlue,citecolor=PineGreen]{hyperref}

\usetikzlibrary{positioning}
\tikzset{main node/.style={circle,fill=black!20,draw,minimum size=1cm,inner sep=0pt},
}

\tikzstyle{vertex}=[circle, draw, fill=black, inner sep=0pt, minimum size=4pt]
\tikzstyle{edge}=[line width=1.5pt,black!50!white]
\tikzstyle{gridp}=[inner sep=1pt,circle,fill=black!70!white]
\tikzstyle{gridl}=[black!50!white]

\tikzstyle{lnode}=[circle,white,draw=black!60!white,fill=black!60!white,inner sep=1pt]
\tikzstyle{cnode}=[circle,draw=black!60!white,fill=black!60!white,inner sep=1.5pt]
\tikzstyle{redge}=[edge,Red]
\tikzstyle{bedge}=[edge,NavyBlue]

\tikzstyle{nvertex}=[vertex, draw=ncol, fill=ncol]
\tikzstyle{edgeq}=[edge,gray!60,densely dashed]
\tikzstyle{nedge}=[edge,ncol]
\tikzstyle{oedge}=[edge,Red!60!black]
\usetikzlibrary{decorations.pathreplacing}

\theoremstyle{plain}

\newtheorem{thm}{Theorem}[section]
\newtheorem{cor}[thm]{Corollary}
\newtheorem{lem}[thm]{Lemma}
\newtheorem{prop}[thm]{Proposition}
\theoremstyle{definition}
\newtheorem{defi}[thm]{Definition}
\newtheorem{defis}[thm]{Definitions}
\newtheorem{conj}[thm]{Conjecture}
\newtheorem{conv}[thm]{Convention}
\newtheorem{nota}[thm]{Notation}
\newtheorem{rem}[thm]{Remark}
\newtheorem{rems}[thm]{Remarks}
\newtheorem{exa}[thm]{Example}
\newtheorem{exas}[thm]{Examples}
\newtheorem{prob}[thm]{Problem}
\newtheorem{probs}[thm]{Problems}
\newtheorem{ques}[thm]{Question}
\newtheorem{sit}[thm]{}

\newcommand{\Spec}{ \operatorname{{\rm Spec}}}
\newcommand{\Frac}{ \operatorname{{\rm Frac}}}

\newcommand{\Div}{ \operatorname{{\rm Div}}}

\newcommand{\Aut}{ \operatorname{{\rm Aut}}}

\newcommand{\ML}{ \operatorname{{\rm ML}}}

\def\deg{\mathop{\rm deg}}

\def\ord{\mathop{\rm ord}}

\def\tp{\mathop{\rm tp}}

\def\ML{\mathop{\rm ML}}
\def\Pic{\mathop{\rm Pic}}

\def\card{\mathop{\rm card}}

\renewcommand{\epsilon}{\varepsilon}

\def\and{\quad\mbox{and}\quad}

\renewcommand{\div}{ \operatorname{\rm div}}

\newcommand{\C}{\ensuremath{\mathbb{C}}}

\newcommand{\Q}{\ensuremath{\mathbb{Q}}}
\newcommand{\Z}{\ensuremath{\mathbb{Z}}}
\newcommand{\N}{\ensuremath{\mathbb{N}}}
\newcommand{\G}{\ensuremath{\mathbb{G}}}

\newcommand{\hX}{{\hat X}}

\newcommand{\cB}{{\ensuremath{\mathcal{B}}}}
\newcommand{\cL}{{\ensuremath{\mathcal{L}}}}

\newcommand{\cF}{{\ensuremath{\mathcal{F}}}}

\newcommand{\cS}{{\ensuremath{\mathcal{S}}}}

\newcommand{\cO}{{\ensuremath{\mathcal{O}}}}

\newcommand{\cD}{{\ensuremath{\mathcal{D}}}}

\newcommand{\cP}{{\ensuremath{\mathcal{P}}}}

\newcommand{\cX}{{\ensuremath{\mathcal{X}}}}
\newcommand{\cY}{{\ensuremath{\mathcal{Y}}}}

\newcommand{\cU}{{\ensuremath{\mathcal{U}}}}

\newcommand{\p}{\partial}

\newcommand{\id}{{\rm id}}

\newcommand{\h}{{\rm ht}}

\renewcommand{\rho}{\varrho}

\def\bals#1\eals{\begin{align*}#1\end{align*}}
\def\bal#1\eal{\begin{align}#1\end{align}}

\def\SAut{\mathop{\rm SAut}}

\def\A{{\mathbb A}}

\def\ZZ{{\mathbb Z}}

\def\PP{{\mathbb P}}

\renewcommand{\phi}{\varphi}

\newcommand{\bnum}{\begin{enumerate}}
\newcommand{\enum}{\end{enumerate}}
\renewcommand{\emptyset}{\varnothing}

\newcommand{\la}{\label}

\newcommand{\brem}{\begin{rem}}
\newcommand{\brems}{\begin{rems}}
\newcommand{\erem}{\end{rem}}
\newcommand{\erems}{\end{rems}}
\newcommand{\bprob}{\begin{prob}}
\newcommand{\eprob}{\end{prob}}
\newcommand{\bprobs}{\begin{probs}}
\newcommand{\eprobs}{\end{probs}}
\newcommand{\bques}{\begin{ques}}
\newcommand{\eques}{\end{ques}}
\newcommand{\bexa}{\begin{exa}}
\newcommand{\bexas}{\begin{exas}}
\newcommand{\eexa}{\end{exa}}
\newcommand{\eexas}{\end{exas}}
\newcommand{\bdefi}{\begin{defi}}
\newcommand{\edefi}{\end{defi}}
\newcommand{\bdefis}{\begin{defis}}
\newcommand{\edefis}{\end{defis}}
\newcommand{\bcor}{\begin{cor}}
\newcommand{\ecor}{\end{cor}}
\newcommand{\blem}{\begin{lem}}
\newcommand{\elem}{\end{lem}}
\newcommand{\bconv}{\begin{conv}}
\newcommand{\econv}{\end{conv}}
\newcommand{\bconj}{\begin{conj}}
\newcommand{\econj}{\end{conj}}
\newcommand{\bprop}{\begin{prop}}
\newcommand{\eprop}{\end{prop}}
\newcommand{\bthm}{\begin{thm}}
\newcommand{\ethm}{\end{thm}}
\newcommand{\bnota}{\begin{nota}}
\newcommand{\enota}{\end{nota}}
\newcommand{\bsit}{\begin{sit}}
\newcommand{\esit}{\end{sit}}
\newcommand{\be}{\begin{equation}}
\newcommand{\ee}{\end{equation}}
\newcommand{\bproof}{\begin{proof}}
\newcommand{\eproof}{\end{proof}}
\def\ba{\begin{array}}
\def\ea{\end{array}}

\thanks{A part of this work
was done during a stay of the authors at the Max Planck Institut
f\"ur Mathematik at Bonn. The authors thank this
institution for its hospitality, support, and excellent working conditions.}

\begin{document}
\title[Cancellation for surfaces]{Cancellation for surfaces revisited. II}

\author{
H.\ Flenner, 
S.\ Kaliman,
M.\ Zaidenberg}
\address{
Fakult\"at f\"ur
Mathematik, Ruhr Universit\"at Bochum, Geb.\ NA 2/72,
Universit\"ats\-str.\ 150, 44780 Bochum, Germany}
\email{Hubert.Flenner@rub.de}
\address{Department of Mathematics,
University of Miami, Coral Gables, FL 33124, USA}
\email{kaliman@math.miami.edu}
\address{Universit\'e Grenoble Alpes, CNRS, Institut Fourier, F-38000 Grenoble, France
}
\email{Mikhail.Zaidenberg@ujf-grenoble.fr}

\begin{abstract} Let  $X$ and $X'$ be affine algebraic varieties over a field $\mathbb{k}$. The celebrated Zariski Cancellation Problem asks as to when the existence of an isomorphism $X\times\A^n\cong X'\times\A^n$  implies $X\cong X'$. In Part I of this paper (arXiv:1610.01805) we provided a criterion for cancellation 
in the case where $X$ is a normal affine surface admitting an $\A^1$-fibration $X\to B$ over a smooth affine curve $B$. If $X$ does not admit such an $\A^1$-fibration then the cancellation by the affine line is known to hold for $X$ by a result of Bandman and Makar-Limanov. 
In the present Part II we classify all pairs $(X,X')$ of smooth affine surfaces $\A^1$-fibered over $B$ with only reduced fibers whose cylinders $X\times\A^1$, $X'\times\A^1$ are isomorphic  over $B$. Our criterion of isomorphism of cylinders over $B$  is expressed in terms of linear equivalence of certain divisors on the  Danielewski-Fieseler quotient of $X$ over $B$. We construct a course moduli space of such surfaces with a given cylinder. This moduli space has an infinite number of irreducible components whose dimension grows. Each irreducible component of this moduli space 
is an affine variety with quotient singularities. 
\end{abstract}
\maketitle

\thanks{
{\renewcommand{\thefootnote}{} \footnotetext{ 2010
\textit{Mathematics Subject Classification:}
14R20,\,32M17.\mbox{\hspace{11pt}}}}

{\footnotesize \tableofcontents}

\section*{Introduction}

Let $X$ and $Y$ 
be algebraic varieties  over a field $\mathbb{k}$. 
The biregular version of the celebrated Zariski Cancellation Problem asks
as to when the existence of a biregular isomorphism 
$X\times\A^n\cong Y\times\A^n$ implies that $X\cong Y$ where $\A^n$ stands for the affine $n$-space over $\mathbb{k}$. 

In the sequel the base field $\mathbb{k}$ is assumed to be algebraically closed of characteristic zero. 
We say that $X$ is a {\em Zariski factor} (a {\em Zariski $1$-factor}, respectively) if 
$X\times\A^n\cong Y\times\A^n$ implies that $X\cong Y$ whatever is $n\in\N$ (for $n=1$, respectively). 

This paper is Part II of \cite{FKZ-I}.  
In \cite{FKZ-I} (referred to as Part I) we obtained the following criterion.

\bthm[Thm.\ I.0.3 \footnote{A reference  "Theorem I.x.y" means "Theorem x.y in \cite{FKZ-I}".}\label{thm: Zariski $1$-factor}] Let $X$ be a normal affine  surface $\A^1$-fibered over a smooth affine curve. 
Then the following conditions are equivalent:
\begin{itemize}
\item[{\rm (i)}] $X$ is a Zariski factor;
\item[{\rm (ii)}]  $X$ is a Zariski $1$-factor; 
\item[{\rm (iii)}] $X$ is a parabolic $\G_m$-surface. 
\end{itemize}
\ethm

Recall (see e.g., \cite{FZ}) that a {\em parabolic $\mathbb{G}_m$-surface} is a normal affine surface equipped with an $\A^1$-fibration $\pi\colon X\to C$ over a smooth affine curve $C$ and an effective $\G_m$-action along the fibers of $\pi$. 

From Theorem \ref{thm: Zariski $1$-factor} and a result of  Bandman and Makar-Limanov (\cite[Lem.\ 2]{BML3}) one deduces such a corollary.

\bcor[Cor.\ I.0.4]\label{cor: characterization-1-factors} A normal affine 
surface  $X$ is a Zariski $1$-factor if and only if either $X$ does not admit any  $\A^1$-fibration over an affine base
or $X$ is a  parabolic $\G_m$-surface. 
\ecor

The present Part II deals with normal $\A^1$-fibered affine surfaces $\pi\colon X\to B$ over a smooth affine curve $B$. If 
$\pi$ has only reduced fibers we call such a surface a \emph{generalized Danielewski-Fieseler surface}, or a \emph{GDF surface} for short. To a GDF  surface $\pi\colon X\to B$ one associates a non-separated one-dimensional scheme ${\rm DF}(\pi)$ called the \emph{Danielewski-Fieseler quotient} along with a surjective morphism ${\rm DF}(\pi)\to B$ and an anti-effective divisor $\tp.\div(\pi)$ on ${\rm DF}(\pi)$ called the \emph{type divisor} (see Definitions \ref{def:DF-quotient} and \ref{def: type}). Our main result is the following theorem.

\bthm\la{thm: main0} Let $\pi\colon X\to B$ and $\pi'\colon X'\to B$ be  GDF surfaces over a smooth affine  curve $B$. Then the cylinders $X\times\A^1$ and $X'\times\A^1$ are isomorphic over $B$ if and only if there exists an isomorphism $\tau\colon{\rm DF}(\pi)\stackrel{\cong_B}{\longrightarrow}{\rm DF}(\pi')$ defined over $B$ such that the divisors $\tp.\div(\pi)$ and $\tau^*(\tp.\div(\pi'))$ on ${\rm DF}(\pi)$ are linearly equivalent.
\ethm

The next corollary follows immediately by using a suitable base change.

\bcor\la{cor: main0} An isomorphism $\varphi\colon X\times\A^1\stackrel{\cong}{\longrightarrow} X'\times\A^1$ which sends the fibers of $X\times\A^1\to B$ to fibers of $X'\times\A^1\to B$ does exist if and only if there exists an isomorphism $\tau\colon{\rm DF}(\pi)\stackrel{\cong}{\longrightarrow}{\rm DF}(\pi')$ such that $\tp.\div(\pi)\sim\tau^*(\tp.\div(\pi'))$.
\ecor

\brems\la{rem: main0} 1. Notice that if $B\not\cong\A^1$ then any isomorphism of cylinders $X\times\A^1\stackrel{\cong}{\longrightarrow} X'\times\A^1$ sends  the fibers to fibers inducing an automorphism of $B$, cf.\ Lemma I.6.10.  

2. It is worth to mention also the following facts. Consider a pair $(\breve{B}\to B,\breve{D})$ where $\breve{D}$ is an anti-effective divisor on 
a  one-dimensional scheme $\breve{B}$ equipped with a  surjective morphism $\breve{B}\to B$. Then  there exists a GDF surface  $\pi\colon X\to B$ such that ${\rm DF}(\pi)=_B\breve{B}$ and $\tp.\div(\pi)=\breve{D}$. Given a pair $(\breve{B}\to B,\breve{D})$ the corresponding GDF surfaces $X$ can vary in non-isotrivial families. However, due to Theorem \ref{thm: main0} the cylinders over these surfaces are all isomorphic over $B$. Moreover, up to an isomorphism over $B$ these cylinders depend only on the class of $\breve{D}$ in the Picard group $\Pic(\breve{B})$. The variation of $\breve{D}$ within its class adds, in general, extra discrete parameters to the isomorphism type of the corresponding GDF surface $X$, see, e.g., Lemma I.7.11 and Corollary I.7.12. 
\erems

The proof of Theorem \ref{thm: main0} is done in Sections \ref{sec:bushes}--\ref{sec:proof}.  
In Section \ref{sec:prelim} we remind some terminology and necessary facts from Part I. In Section \ref{sec: reduction-dual-graph} we provide a new proof of a result of Bandman and Makar-Limanov  which gives a sufficient condition for almost flexibility of the cylinder over an $\A^1$-fibered affine surface, see Theorem \ref{cor: Giz-cylinder}. In the concluding Section \ref{sec:deformations} we construct, under a certain restriction, a coarse moduli space of marked GDF surfaces with the given base and  graph divisor, see Theorem \ref{thm:moduli}. The cylinders over these surfaces are all isomorphic. A simple Example \ref{ex:non-algebraic} shows that without our restriction, the coarse moduli space of such surfaces does not exist, in general. The irreducible components of the moduli space of GDF surfaces with a given cylinder have unbounded dimensions, in general. 

\section{Preliminaries}\la{sec:prelim}
We recall  some definitions and results from Part I that are used in the sequel.

\bdefi[\emph{GDF surface}]\label{def: affine A1} Let $X$ 
be a normal affine surface over $\mathbb{k}$. A morphism
$\pi\colon  X\to B$  onto a smooth affine curve $B$ is called an {\em  $\A^1$-fibration} 
if the fiber $\pi^*(b)$ over a general point $b\in B$ is isomorphic to the affine line $\A^1$. Each fiber of an $\A^1$-fibration is a disjoint union of curves isomorphic to $\A^1$, see \cite[Ch.\ 3, Lem.\ 1.4.1]{Mi}. 

An $\A^1$-fibered surface $\pi\colon X\to B$ is called a 
{\em generalized Danielewski-Fieseler surface}, or a {\em GDF surface} for short, 
if the fibers $\pi^*(b)$, $b\in B$, 
are reduced. Any GDF surface is smooth (see, e.g.,  \cite[(2.2)]{Du0} and Lemma I.2.18(b)).
\edefi

\bexa\label{line-bdls} If all the fibers of  a GDF surface $\pi\colon  X\to B$ are irreducible then $\pi$ is the projection of a locally trivial $\A^1$-fiber bundle. Since the base $B$ is affine this fiber bundle admits regular sections.
Given such a section $Z$ there is a unique structure of a line bundle, say, $L$ on $\pi\colon  X\to B$ with zero section $Z$. Let  $Z'\neq Z$ be a second regular section of $L$, and let $L'$ be the line bundle over $B$ with projection $\pi$ and zero section $Z'$. Since the effective zero divisor $\div_Z(Z')=\div_{Z'}(Z)$ belongs to the classes of  both $L$ and of $L'$ in $\Pic B$, these classes coincide, that is,  $L\cong_B L'$. 

 Notice that any class in $\Pic B$ contains effective and anti-effective divisors. Via the above procedure any such class can be represented by an $\A^1$-fiber bundle $X\to B$, that is, by a GDF surface with only irreducible fibers. 
\eexa

\bdefi[\emph{Trivializing sequence}]\label{def: trivializing-seq}  Given a GDF surface $\pi\colon X\to B$ there is a sequence of GDF surfaces $\pi_i\colon X_i\to B$ and a sequence
{\rm (I.9)} of birational $B$-morphisms
\be\la{eq: trivializing-seq-aff-modif} 
X =  X_m \stackrel{\varrho_m}{\longrightarrow} X_{m-1} \stackrel{}{\longrightarrow}  
\ldots \stackrel{}{\longrightarrow}   X_1 \stackrel{\varrho_1}{\longrightarrow} 
X_0=B\times\A^1\,
\ee where $\rho_i\colon X_i\to X_{i-1}$ is  a \emph{fibered modification}, that is, an affine  modification over $B$ with a reduced zero-dimensional center. This sequence can be chosen so that
 the center of $\rho_i$ is contained in the exceptional divisor of $\rho_{i-1}$ for every $i\ge 2$.
Let $\bar B$ be a smooth completion of $B$.
There exist a sequence of ruled surfaces $\bar\pi_i\colon\bar X_i\to \bar B$ and a sequence of birational morphisms over $\bar B$ (see
{\rm (I.10)})
\be\la{eq: trivializing-seq-blowups} 
\bar X =  \bar X_m \stackrel{\bar\varrho_m}{\longrightarrow} \bar X_{m-1} \stackrel{}{\longrightarrow}  
\ldots \stackrel{}{\longrightarrow}   \bar X_1 \stackrel{\bar\varrho_1}{\longrightarrow} 
\bar X_0=\bar B\times\PP^1\,
\ee
which extends \eqref{eq: trivializing-seq-aff-modif}. In particular, 
$\bar X_i$ is a
smooth completion of $X_i$  by an SNC divisor $D_i\subset \bar X_i$
such that
\begin{itemize}\item there is a unique component $S_i$ of  $D_i$ which is a section of $\bar\pi_i\colon\bar X_i\to \bar B$; the remaining components of $D$ are fiber components;
\item $\bar\rho_1\colon\bar X_1\to \bar X_0$ is a blowup with center a reduced zero dimensional subscheme in $\bar X_0\setminus S_0$ supported in the fibers over some points $b_1,\ldots,b_n\in B$, $n\ge 0$; 
\item for $i>1$, $\bar\rho_i\colon\bar X_i\to\bar X_{i-1}$ is the blowup of a reduced zero dimensional subscheme of $X_{i-1}$ supported on the exceptional divisor of $\bar\rho_{i-1}$. 
\end{itemize}
\edefi

\bconv\label{conv0} Given a GDF surface $\pi\colon X\to B$, the trivializing sequences \eqref{eq: trivializing-seq-aff-modif} and \eqref{eq: trivializing-seq-blowups} are not unique, in general. We fix an arbitrary such pair of sequences. In particular, we fix a completion $\bar\pi\colon\bar X\to\bar B$ of $\pi\colon X\to B$. By abuse of notation, we attribute to $\pi$ the objects  depending a priori on $\bar\pi$. \econv

\bdefi[\emph{Markings and special fibers}]\la{def:marking} A \emph{marked GDF surface} is a GDF surface $\pi\colon X\to B$ with a given \emph{marking} $z\in\cO(B)\setminus\{0\}$ where $z$ has only simple zeros and vanishes at the points $b_1,\ldots,b_n$ as in \ref{def: trivializing-seq}. 
It is convenient to enlarge the set $\{b_1,\ldots,b_n\}$ including all the zeros of $z$, so that $b_1+\ldots+b_n=z^*(0)$ is a principal reduced effective divisor on $B$. Notice that the projection $\pi\colon X\to B$ restricted over the Zariski open set $B^*:=B\setminus\{b_1,\ldots,b_n\}$ defines a trivial line bundle, a trivialization being given by the composition $\sigma:=\rho_m\circ\ldots\circ\rho_1:X\to B\times\A^1$ of morphisms in \eqref{eq: trivializing-seq-aff-modif} restricted over $B^*$.
The fibers $\pi^{-1}(b_i)$, $i=1,\ldots,n$ are called \emph{special}.
\edefi

\bdefi[\emph{Fiber trees, levels, and types}]\label{def: fiber tree}  Given a marked GDF surface $X\to B$, its SNC 
completion $\bar\pi\colon\bar X\to\bar B$, and a point $b\in B$,  the dual graph  $\Gamma_b=\Gamma_b(\bar\pi)$ of the fiber $\bar\pi^{-1}(b)$ is a rooted tree called a  \emph{fiber tree}.  It depends on the completion chosen. 
The  root $v_0\in\Gamma_b$ is the  neighbor of the section $S\subset \bar X\setminus X$ in the  dual graph $\Gamma_{\rm ext}$ of the \emph{extended divisor} $$D_{\rm ext}=S+\bar\pi^{-1}(\bar B\setminus B)+\sum_{j=1}^n \bar\pi^{-1}(b_j)\,$$ where $b_1,\ldots,b_n\in B$ are as in Definition \ref{def:marking}. A vertex $v$ of $\Gamma_b$ has \emph{level} $l$ if the tree distance between $v$ and $v_0$ equals $l$. Thus,  $v_0$ is a unique vertex of $\Gamma_b$ on level $0$.
By a \emph{height} $\h(\Gamma_b)$ we mean the highest level of the vertices in $\Gamma_b$. The \emph{leaves} of a rooted tree are its extremal vertices different from the root. A fiber component of $\bar\pi^{-1}(b_j)$ which corresponds to a leave is called a \emph{feather}. In fact, the feathers are the closures in $\bar X$ of the fiber components of $\pi^{-1}(b_j)\subset X$, $j=1,\ldots,n$. Thus, $D_{\rm ext} - D$ is a union of the feathers of $D_{\rm ext}$. By the \emph{type} $\tp(\Gamma_b)$ we mean the sequence of non-negative integers $(n_0,n_1,\ldots, n_h)$ where $h=\h(\Gamma_b)$ and  $n_i$ is the number of leaves of $\Gamma_b$ on level $i$. 
For instance, if $\h(\Gamma_b)>0$  and $\tp(\Gamma_b)=(0,\ldots,0,1)$ then $\Gamma_b$ is a chain $[[-1,-2,\ldots,
-2,-1]]$.\edefi

\brem\label{well-ordered-seq}
The morphism $\bar\rho_i\colon \bar X_l\to \bar X_{l-1}$ in \eqref{eq: trivializing-seq-blowups} contracts the  fiber components of $D_{l,\rm ext}$ on the top level $l$ into distinct points on a union of fiber components of $D_{l-1,\rm ext}$ of top level $(l-1)$. These points  are smooth points of $D_{l-1,\rm ext}$. 
\erem

\bdefi[\emph{Graph divisors}] A \emph{graph divisor} on $B$ is a formal sum $\cD=\sum_{i=1}^n \Gamma_i b_i$ where $b_1,\ldots,b_n\in B$ are distinct points and $\Gamma_i$, $i=1,\ldots,n$, are rooted trees. 

We say that $\cD$ is a  \emph{chain divisor} if all the $\Gamma_i$ are linear graphs. Let $\pi\colon X\to B$ be a marked GDF surface with a marking $z\in\cO_B(B)$ such that $z^*(0)=b_1+\ldots+b_n$. Then $\pi$ is the projection of a line bundle over $B$ if and only if the graph divisor $\cD(\pi):=\sum_{i=1}^n \Gamma_i(\pi) b_i$ is a chain divisor.

Notice that under the conventions of Definition \ref{def: trivializing-seq}  for every $l=0,\ldots,m$ in  \eqref{eq: trivializing-seq-aff-modif} one has $\cD(\pi_l)=\cD(\pi_m)_{\le l}$ and $\h(\cD(\pi_l))=l$ (cf.\ Remark \ref{well-ordered-seq}.1). 
\edefi

\bdefi[\emph{The Danielewski-Fieseler quotient}]\la{def:DF-quotient} Given a GDF surface $\pi\colon X\to B$ the \emph{Danielewski-Fieseler quotient} ${\rm DF}(\pi)$ is the quotient of $X$ by the equivalence relation defined by the fiber components of $\pi$. Thus, ${\rm DF}(\pi)$ is a (non-separated, in general) one-dimensional  scheme, and $\pi$ factorizes as follows: $$\pi\colon X\stackrel{p}{\longrightarrow} {\rm DF}(\pi)\stackrel{q}{\longrightarrow} B\,$$ where the fibers of $p\colon X\to {\rm DF}(\pi)$  are reduced and irreducible.  
In particular, $q\colon {\rm DF}(\pi)\to B$ is an isomorphism over $B\setminus \{b_1,\ldots,b_n\}$, while the total transform of $b_i$  in ${\rm DF}(\pi)$ consists of $N_i$ points $(b_{i,j})_{j=1,\ldots,N_i}$ where $N_i$ is the number of fiber components 
$F_{i,j}$ in $\pi^{-1}(b_i)$. Therefore, $q$  is an isomorphism if and only if all the fibers of $\pi$ are irreducible, if and only if $\pi\colon X\to B$ admits a structure of a line bundle. 
\edefi

\bdefi[\emph{Type divisors}]\la{def: type}
Let $l_{i,j}=l(F_{i,j})$ 
be the level of $F_{i,j}$. The anti-effective divisor   on ${\rm DF}(\pi)$ $$\tp.\div(\pi)=-\sum_{i,j} l_{i,j}b_{i,j}$$ is called the \emph{type divisor} of $X$. 

The linear equivalence of divisors on ${\rm DF}(\pi)$ is, as usual, the  equivalence modulo the principal divisors on ${\rm DF}(\pi)$. The latter divisors  are just the principal divisors on $B$ lifted to ${\rm DF}(\pi)$. The Picard group $\Pic {\rm DF}(\pi)$ is defined in a usual way. To a GDF surface $\pi\colon X\to B$ we associate its Picard class $[\tp.\div(\pi)]\in\Pic {\rm DF}(\pi)$. If $\pi\colon X\to B$ represents a line bundle $L$ over $B$ then ${\rm DF}(\pi)=B$ and $[\tp.\div(\pi)]=[L]\in\Pic B$.  
\edefi

\brems\la{rem:Pic} 1. It is easily seen that $$\Pic {\rm DF}(\pi)\cong (\Pic B) \oplus \ZZ^{\rho(\pi)}\quad\mbox{where}\quad \rho(\pi)=\sum_{j=1}^n (N_j-1)\,.$$ 

2. Let $\pi\colon X\to B$ be a GDF surface. It is well known that $p\colon X\to {\rm DF}(\pi)$ is an $\A^1$-fiber bundle (see, e.g., Proposition I.3.3). If ${\rm DF}(\pi)$ is non-separated then the latter fiber bundle does not admit any regular section. Indeed, the image of such a section would be a non-separated reduced proper subscheme of $X$. However, the latter is not possible since $X$ is separated. By a similar reason, given a line bundle $L$ over ${\rm DF}(\pi)$,  the total space of $L$ is affine if and only if ${\rm DF}(\pi)$ is separated, that is, ${\rm DF}(\pi)=B$. 

3. Given a trivializing sequence \eqref{eq: trivializing-seq-aff-modif} and  a regular section $s_0\colon B\to X_0=B\times\A^1$ the proper transform $s_m$ of $s_0$ in $X=X_m$ acquires a pole of order $l_{i,j}$ over a point $b_{i,j}\in {\rm DB}(X)$.  The latter  means the following. Choose the standard affine chart $U_{i,j}\cong_{B_i} B_i\times\A^1$ near $F_{i,j}$ with natural local coordinates $(z,u_{i,j})$ where $z$ gives a local parameter on $(B,b_i)$ and $B_i=(B\setminus\{b_1,\ldots,b_n\})\cup\{b_i\}$. Then  at the point $b_i$ the meromorphic function $z^{l_{i,j}}u_{i,j}\circ s_m$ is regular and does not vanish. 
Therefore, $-\div_\infty (s_m)\sim \tp.\div(\pi)$ where the pole divisor $\div_\infty (s_m)$ on $ {\rm DF}(X)$ is defined above. The proof goes by induction on the length $l_{i,j}$ of the shortest path in the fiber tree $\Gamma_{b_i}(\pi)$ joining the leaf $\bar F_{i,j}$ with the root. Each subsequent blowup in \eqref{eq: trivializing-seq-blowups} along this path increases the pole order by one. 
\erems

\bdefi[\emph{$(A,\overline{-1})$-stretching}; cf.\ Definitions I.7.1 and I.7.2]
\la{def: stretching} 
Given an effective divisor $A=\sum_{i=1}^n a_ib_i\in\Div (B)$ where $a_i\in\Z_{\ge 0}$ and $b_i\in B$ we  associate with $A$ a chain divisor $\cD(A)=\sum_{i=1}^n L(a_i)b_i$ where $L(a_i)$ is a chain with weights $[[-1,-2,\ldots,-2,-2]]$ of length $a_i$ if $a_i>0$ and $L(0)=\emptyset$ otherwise. Let $\cD=\sum_{i=1}^n \Gamma_i b_i$ be a  graph divisor.
We let $\cD'=\sum_{i=1}^n \Gamma'_i b_i$ where 
$\Gamma'_{i}$ is obtained from $\Gamma_{i}$ by inserting  the chain $L_i$ below the root $v_i$ of $\Gamma_{i}$
so that the $(-1)$-vertex of $L_i$ becomes the root of $\Gamma'_{i}$ and the right end vertex of $L_i$ 
is joint with $v_{i}$ whose weight changes accordingly.
The transformation $\cD\leadsto\cD'$ is called a \emph{combinatorial $(A,\overline{-1})$-stretching}. A stretching is called \emph{principal} if $A\in\Div(B)$ is a principal divisor.

Given a GDF surface $\pi\colon X\to B$ the \emph{principal geometric  $(A,\overline{-1})$-stretching} with $A=\div f$ for $f\in\cO_B(B)\setminus\{0\}$ amounts to perform in (\ref{eq: trivializing-seq-aff-modif}) the affine modification $X_0'\to X_0=B\times\A^1$ along the divisor $(f\circ\pi)^{*}(0)$ on $X_0$ with center being the ideal $(f,u)\subset\cO_{X_0}(X_0)=\cO_B(B)[u]$. In other words, letting $\A^1=\Spec \mathbb{k}[u]$ one has  
\be\la{eq:-1-stretching} \cO_{X_0'}(X_0')=\cO_B(B)[u']\quad\mbox{where}\quad u'=u/f\,.\ee 
Therefore, $X_0'\cong_B X_0$. Performing the remaining fibered modifications in (\ref{eq: trivializing-seq-aff-modif}) and reorganizing the sequence accordingly gives again the same surface $X'=X_{m'}'=X_m=X$ along with a new trivializing sequence \eqref{eq: trivializing-seq-aff-modif}. Thus, the principal $(A,\overline{-1})$-stretching preserves the initial GDF surface. However, it affects the trivializing completion $(\bar X, D)$ along with sequence  (\ref{eq: trivializing-seq-blowups}) by inserting the chains $L_i$ in the graph divisor between the section $S$ and the roots as described before.  Abusing notation we let $\pi'=\pi_{X'}\colon X'\to B$ (in fact, $\pi'=\pi$).
One has
\be\la{eq:type-div-transform} {\rm tp.div}\,(\pi')={\rm tp.div}\,(\pi)-\div f\,.
\ee
In other words, the type divisor of a GDF surface is defined only up to adding a principal anti-effective divisor.
\edefi

The latter observation leads to the following lemma. 

\blem\la{lem: reduction} Consider two GDF surfaces
$\pi_X\colon X\to B$ and $\pi_Y\colon Y\to B$ with the same Danielewski-Fieseler quotient and with linearly equivalent  type divisors, see Definition {\rm \ref{def: type}}. Then one can choose new trivializing sequences \eqref{eq: trivializing-seq-aff-modif} for $X$ and $Y$ in such a way that the corresponding type divisors coincide. 
\elem

\bproof By assumption there is a principal divisor $T\in\Div B$ such that
\be\la{eq:l.e.-type-div} \tp.\div(\pi_X)= \tp.\div(\pi_Y)+T\,.\ee
Let $A\in\Div B$ be a principal effective divisor such that $A-T\ge 0$. Performing the principal $(A,\overline{-1})$-stretching and $(A-T,\overline{-1})$-stretching, respectively, one  replaces the trivializing sequences \eqref{eq: trivializing-seq-aff-modif} for $X$ and $Y$ by suitable new ones so that the new type divisors for $X$ and $Y$ are 
$$\tp.\div(\pi_X)-A,\quad\mbox{resp.,}\quad\tp.\div(\pi_Y)-(A-T)\,,$$ 
see \eqref{eq:type-div-transform}. Due to \eqref{eq:l.e.-type-div} the latter divisors are equal. 
\eproof

\brem\la{rem: 9.1a} Let $\pi\colon  X\to B$ be a marked GDF 
surface  with a marking $z\in\cO_B(B)\setminus\{0\}$ where $z^*(0)=b_1+\ldots+b_n$ is a reduced effective divisor. Performing a principal $(A,\overline{-1})$-stretching with the divisor $A=\div z^d=d(b_1+\ldots+b_n)$ one may suppose that the graph divisor $\cD(\pi)$ does not have any leaf on level $\le d-1$; cf.\ Lemma I.7.3.
\erem

\section{Classification of GDF cylinders up to $B$-isomorphism}
\label{sec: reduction-dual-graph}
Hereafter $B$ stands for a smooth affine curve.  Theorem \ref{thm: main0} can be reformulated as follows. 

\bthm\la{thm: main}
Let $\pi_X\colon X\to B$ and $\pi_Y\colon Y\to B$ be two GDF surfaces over the same base $B$.
Then the cylinders $\cX=X\times\A^1$ and $\cY=Y\times\A^1$ are isomorphic over $B$ if and only if there exists an isomorphism $\tau\colon {\rm DF}(\pi_X)\stackrel{\cong_B}{\longrightarrow}{\rm DF}(\pi_Y)$ such that $[\tp.\div(\pi_X)]=[\tau^*(\tp.\div(\pi_Y))]$
in the Picard group $\Pic {\rm DF}(\pi_X)$.
\ethm

The proof starts with the following elementary lemma.

\blem\la{lem:elem} An isomorphism $\varphi\colon\cX\stackrel{\cong_B}{\longrightarrow}\cY$ induces an isomorphism $\tau\colon {\rm DF}(\pi_X)\stackrel{\cong_B}{\longrightarrow}{\rm DF}(\pi_Y)$. 
\elem

\bproof The Danielewski-Fieseler quotient ${\rm DF}(\pi_X)$ is the quotient of $X$ by the equivalence relation defined by the fiber components of $\pi_X$. It 
coincides with the quotient of the cylinder $\cX$ by the equivalence relation defined by the fiber components of $\cX\to B$. The isomorphism $\varphi\colon\cX\stackrel{\cong_B}{\longrightarrow}\cY$ respects the latter equivalence relations on $\cX$ and $\cY$.
\eproof

This lemma along with the linear equivalence of the type divisors established in Subsection \ref{sec:proof-only-if} gives the `only if' part of Theorem \ref{thm: main}.

The strategy of the proof of the `if' part is as follows. We reduce the assertion to the case where the type divisor completely determines the graph divisor. This is so indeed if the fiber trees are bushes, see Definition \ref{ex: unibranched}. First we obtain this reduction assuming that the GDF surface $\pi_X\colon X\to B$ has a unique special fiber $\pi_X^{-1}(b_0)$, $b_0\in B$, see Theorem \ref{thm: 8.main}. In Subsection \ref{ss:regularization} we treat the general case of GDF surfaces  with any number of special fibers.  
 
Theorem I.5.7 of Part I is crucial for the proof. This theorem says that certain continuous parameters of a GDF surface $\pi_X\colon X\to B$ are irrelevant for the $B$-isomorphism class of the cylinder $\cX=X\times\A^1$. This allows to replace the initial GDF surface $\pi_X\colon X\to B$ by a suitable new one $\pi_{X'}\colon X'\to B$ carrying  the same graph divisor $\cD(\pi_{X'})=\cD(\pi_X)$ and admitting a quite simple explicit description. 

Let us indicate some corollaries of Theorem \ref{thm: main}. First of all, the GDF surfaces over a given smooth affine curve $B$ which are not Zariski 1-factors and whose cylinders are isomorphic over $B$ to a given one form a countable number of families with affine bases. The dimensions of the bases are unbounded. This follows also from Theorem I.5.7 and Proposition I.7.8. A similar conclusion holds as well for the collection of all normal affine surfaces $\A^1$-fibered over a given smooth affine curve and having a given cylinder.

The following result confirms  Conjecture I.1.1
of Part I under a certain additional assumption. It is due to Bandman and Makar-Limanov (\cite[Thm.\ 1]{BML3}).

\bthm[{\rm Bandman-Makar-Limanov}]\label{cor: Giz-cylinder} Let $\pi_X\colon X\to \A^1$ be a Danielewski-Fieseler surface, that is, a GDF surface over $B=\Spec\mathbb{k}[z]$ with the unique special fiber $z^*(0)$. Then  there is an isomorphism of cylinders $\cX\cong_B\cY$ where $\cY=Y\times\A^1$  is  the cylinder over a Gizatullin GDF surface $\pi_Y\colon Y\to\A^1$.  Hence the group $\SAut\cX$ acts in $\cX$ with a 
Zariski open orbit $O$ such that ${\rm codim}_X(X\setminus O)\ge 2$. 
\ethm

An elegant direct proof of this theorem in \cite{BML3} exploits the Danielewski construction. We provide here an alternative argument based on Theorem \ref{thm: main}.

\bproof  The assertion is trivial if $X\cong\A^2$. To exclude this case we will assume that the fiber $\pi_X^{-1}(0)$ with the fiber tree $\Gamma_0=\Gamma_0(\pi_X)$  is reducible. Letting $\tp(\Gamma_0)=(n_i)_{i=0,\ldots,h}$ by our assumption $n_0=0$ and $h={\rm ht}(\Gamma_0)>0$. 

By Theorem \ref{thm: main} it suffices to find a Gizatullin GDF surface $\pi_Y\colon Y\to\A^1$ with a unique special fiber $\pi_Y^{-1}(0)$ such that $\tp(\Gamma_0(\pi_Y))=\tp(\Gamma_0)$.
 Consider a chain $\gamma_0=[v_0,\ldots,v_{h-1}]$ of length $h-1$. Let  $\Gamma_0'$ be the rooted tree with  $v_0$ for the root, $\gamma_0$ for the trunk, and with the leaves $v_{i+1,j}$, $j=1,\ldots,n_{i+1}$ on level $i+1$,  $i=0,\ldots,h-1$,  joint with $\gamma_{0}$ by the edges $[v_{i},v_{i+1,j}]$. 
So,  $\Gamma_0'$ has exactly $n_{i}$ leaves $v_{i,j}$ on level $i$. Therefore,  $\tp(\Gamma_0')=\tp(\Gamma_0)$.

Let $\pi_Y\colon Y\to\A^1$ be a Danielewski-Fieseler surface such that $\Gamma_0(\pi_Y)=\Gamma_0'$. By Theorem \ref{thm: main} one has $\mathcal{X}\cong_B\mathcal{Y}$. Clearly, $Y\not\cong\A^1\times\A^1_*$ is a Gizatullin surface. Indeed,  there exists an SNC completion $(\bar Y, D)$ of $Y$ with the linear dual graph $\Gamma(D)=[F_\infty, S, v_0,\ldots,v_{h-1}]$ obtained by attaching to $\gamma_0$ the $[[0,0]]$-chain $[F_\infty, S]$. Thus, $Y$ is a Gizatullin surface. 

The group $\SAut Y$ acts on $Y$ with a Zariski  open orbit whose complement is finite (\cite{Gi}). Hence also $\SAut \cY\supset \SAut Y\times \SAut \A^1$ has a Zariski open orbit in $\cY\cong\cX$ with a complement of codimension at least 2. In particular, $\ML(\cX)\cong \ML(\cY)$ is trivial. 
 \eproof

\section{GDF surfaces whose fiber trees are bushes} \label{sec:bushes} 
In this section we assume for simplicity that $B$ is the affine  line $\A^1$ and $\pi\colon X\to\A^1$ has a unique special fiber $\pi^{-1}(0)$. Later on we will indicate the modifications which allow to treat the general case. 

\bdefi[\emph{Bushes}]\label{ex: unibranched}  A  \emph{bush} is a rooted tree $\Gamma$
such that the branches of $\Gamma$ at the root $v_0$ are chains, see Figure \ref{fig:bush}. If $\tp(\Gamma)=(n_i)_{i\ge 0}$, see \ref{def: fiber tree},  then $\Gamma$ has exactly $n_i$ branches of height $i$, the root $v_0$ of $\Gamma$ being the  common tip of each branch. 
\edefi

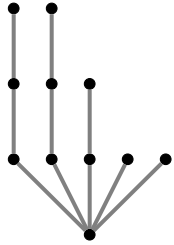
\begin{figure}[h]
\begin{tabular}{l@{\hspace{70pt}}l@{\hspace{70pt}}l}
    \pgfdeclarelayer{background}
		\pgfsetlayers{background,main}
    \begin{tikzpicture}[scale=2]

 \node[vertex] (a1) at (0.5,0) {};
      \node[vertex] (b11) at (0,0.5) {};
      \node[vertex] (b21) at (0.25,0.5) {};
      \node[vertex] (b31) at (0.5,0.5) {};
      \node[vertex] (b41) at (0.75,0.5) {};
      \node[vertex] (b51) at (1,0.5) {};
       \node[vertex] (c11) at (0,1) {};
       \node[vertex] (c21) at (0.25,1) {};
       \node[vertex] (c31) at (0.5,1) {};
       \node[vertex] (d11) at (0,1.5) {};
       \node[vertex] (d21) at (0.25,1.5) {};

      \draw[edge] (a1)edge(b11) (a1)edge(b21) (a1)edge(b31) (a1)edge(b41) (a1)edge(b51) (b11)edge(c11) (b21)edge(c21) (b31)edge(c31) (c11)edge(d11)  (c21)edge(d21) ;

         \end{tikzpicture}
\end{tabular}
  \caption{A bush $\Gamma$ of height $3$ and of type ${\rm tp}(\Gamma)=(0,2,1,2)$}
  \la{fig:bush}
\end{figure}

\bdefi[\emph{Accompanying  sequence of a bush}]\label{def: adopted} 
Let $\Gamma$ be a bush  of height $m>0$. We associate with $\Gamma$ a sequence $(p_i, r_i)_{i=1,\ldots,m}\in (\mathbb{k}[u])^{2m}$ of pairs of monic polynomials with simple roots  such that
\begin{itemize}\item $\deg p_i$ equals the number of vertices of $\Gamma$ on level $i$;  \item $p_{i+1}|p_i$ for $i=1,\ldots,m-1$;\item $r_ip_i=p_1$ for $i=1,\ldots,m$. \end{itemize} 
 Thus, $r_1=1$ and $r_i|r_{i+1}$ for $i=1,\ldots,m-1$.  Furthermore, $\tp(\Gamma)=(n_i)_{i\ge 0}$ where $n_i=\deg p_{i}-\deg p_{i+1}$; see Definition \ref{def: fiber tree}. For instance, for $\Gamma$ in Figure \ref{fig:bush} one can choose 
\be\la{eq:accomp-sec} p_1=\prod_{i=0}^4(u-i),\quad p_2 = \prod_{i=0}^2(u-i),
\quad p_3=\prod_{i=0}^1(u-i),\quad\mbox{and}\quad r_i=p_1/p_i,\,\,\,i=1,2,3\,.
\ee

To any root $\alpha$ of $p_1$ there corresponds a unique branch $\cB_\alpha$ of $\Gamma$ of height $l(\alpha):=\h(\cB_\alpha)\ge 1$ where $l(\alpha)$ satisfies
\be\label{eq: order} p_1(\alpha)=p_2(\alpha)=\ldots=p_{l(\alpha)}(\alpha)=0\quad\mbox{and}\quad p_{l(\alpha)+1}(\alpha) \neq 0\,\ee or, which is equivalent, 
\be\label{eq: order-1} r_1(\alpha)\neq 0,\ldots, r_{l(\alpha)}(\alpha)\neq 0,\quad r_{l(\alpha)+1}(\alpha)=\ldots=r_m(\alpha)=0\,.\ee
\edefi

\bexa\la{ex: Dani-bush} Consider the surface $X_m=\{z^mt-p(u)=0\}$ in $\A^3$ of Danielewski type with the projection $\pi_{m}\colon X_m\to\A^1$, $(z,u,t)\mapsto z$.
Then $\Gamma_0(\pi_{m})$ is a bush with $d=\deg p$ branches  of equal length $m$ and $(p_i,r_i)=(p,1)$ for $i=1,\ldots,m$;  cf.\ Examples I.3.10, I.3.11, and I.7.10. 
\eexa

\brem\la{rem: mu-bush}  The accompanying sequence $(p_i, r_i)_{i=1,\ldots,m}$ of $\Gamma$ is uniquely determined by the pair $(\Gamma, p_1)$.
Given $\Gamma$ this family of polynomials is parameterized by the coefficients of $p_1$, that is, by the points in $\A^{n} \setminus D_n$ where $n=\deg p_1$ and $D_n=\{{\rm discr}(p_1)=0\}$ is the discriminant hypersurface in $\A^{n}$.
\erem

Hereafter we adopt the following convention. 

\bconv[\emph{Reduction of the base field}]\la{conv-C}  Given a finite collection of affine varieties over $\mathbb{k}$ defined by systems of polynomial equations in $\A^N$ one can replace the base field $\mathbb{k}$ by the finite extension $\Q\subset \mathbb{k}' $ generated in $\mathbb{k}$ by the coefficients of all these polynomials. Choosing an embedding $\mathbb{k}'\hookrightarrow\C$ one may  assume that $\mathbb{k}=\C$. 
\econv

In the next proposition we associate to any bush $\Gamma_0$ along with an accompanying
sequence $(p_i,r_i)$  a certain Danielewski-Fieseler surface $\pi\colon X\to\A^1$ in  $\A^N$ such that $\Gamma_0(\pi)=\Gamma_0$; cf.\  \cite{Du0, Du1, DP0}.

\bprop\la{9.1} Fix a bush $\Gamma_0$ of height $m>0$ and an accompanying sequence  $(p_i, r_i)_{i=1,\ldots,m}$ as in {\rm \ref{def: adopted}}.  For $1\le j\le m$ consider the subvariety  $W_j\subset\A^{j+2}=\Spec\mathbb{k}[z,u,t_1, \ldots ,t_j]$ given by 
\be\label{eq: bush} zt_1-p_1(u)=0,\quad zt_i-r_i(u)t_{i-1}=0,\quad i=2,\ldots,j\,.\ee
Then the following hold. 
\begin{itemize}
\item[{\rm (i)}] There is a unique irreducible component $X_j$ of $W_j$ which dominates the $z$-axis; 
\item[{\rm (ii)}]  $\pi_j:=z|_{X_j}\colon X_j\to \A^1$ is a GDF surface with the unique special fiber $\pi_j^{-1}(0)$; 
\item[{\rm (iii)}] $(X_j)_{j=1,\ldots,m}$ fits in a sequence \eqref{eq: trivializing-seq-aff-modif}
of fibered modifications 
\be\la{eq: seq-aff-modif-final} 
X_m \stackrel{\varrho_m}{\longrightarrow} X_{m-1} \stackrel{}{\longrightarrow}  
\ldots \stackrel{}{\longrightarrow}   X_1 \stackrel{\varrho_1}{\longrightarrow} 
X_0=\A^2=\Spec \mathbb{k}[z,u]\,
\ee 
where 
$$\rho_i\colon (z,u,t_1, \ldots ,t_i)\mapsto (z,u,t_1, \ldots ,t_{i-1})\,,\quad i=1,\ldots,m\,;$$
\item[{\rm (iv)}] 
$\Gamma_0(\pi_j)
$ is the bush $\Gamma_0$ truncated at the level $j$. In particular,  $\Gamma_0(\pi_m)=\Gamma_0$.  
\end{itemize}
\eprop

\bproof  
Consider  the hyperplane $L_0=\{z=0\}$ in $\A^{j+2}$. The projection 
\be\la{eq:pi} \pi\colon \A^{j+2}\to\A^{2}, \quad (z,u,t_1, \ldots ,t_j)\mapsto (z,u)\,\ee restricted to $W_j\setminus L_0$ yields an isomorphism $W_j\setminus L_0\cong\A^1_*\times\A^1$. The closure $X_j:=\overline{ W_j\setminus L_0}$ is a surface satisfying (i) such that $\pi_j|_{ X_j\setminus L_0}\colon {X_j\setminus L_0}\to\A^1_*$ is induced by the first projection of $\A^1_*\times\A^1$. Any component of $W_j$ different from $X_j$ is of the form $\{(0,\alpha)\}\times\A^j$ where $\alpha$ is a root of $p_1$ and $\A^j=\Spec \mathbb{k}[t_1,\ldots,t_j]$.

Let us introduce the following notation. Given a root $\alpha$ of $p_1$ consider the open set $\omega(\alpha)\subset\A^1=\Spec \mathbb{k}[u]$ given by $p_1(u)/(u-\alpha)\neq 0$. Let $U_0(\alpha)=\A^1\times\omega(\alpha)\subset X_0=\Spec \mathbb{k}[z,u]$, and for $j=1,\ldots,m$ let $U_j(\alpha)=\pi^{-1}(U_0(\alpha))\subset X_j$ where $\pi$ is as in \eqref{eq:pi}. The collection $\{U_j(\alpha)\}_{p_1(\alpha)=0}$ gives an open covering of $X_j$. Restricting to $U_j(\alpha)$ amounts to passing to the localization 
$$\cO_{U_j(\alpha)}(U_j(\alpha))=\cO_{X_j}(X_j)[(u-\beta)^{-1} | p_1(\beta)=0,\beta\neq\alpha]\,.$$ The projection $\pi_j|_{U_j(\alpha)}\colon U_j(\alpha)\to \A^1=\Spec \mathbb{k}[z]$ is dominant and defines an $\A^1$-fibered surface such that $U_j(\alpha)\setminus L_0=X_j\setminus L_0$. 

\smallskip

\noindent {\bf Claim.} \emph{Let {\rm (\ref{eq: order-1})} holds for a root $\alpha$ of $p_1$ with $l(\alpha)=l$. Then the pair $(z,t_j-t_j(x))$ for $j\in\{1\ldots,l\}$ and $(z,t_l-t_l(x))$ for $j\in\{l,\ldots,m\}$, respectively, is a local analytic coordinate system near any point $x\in X_j$ such that $z(x)=0$ and $u(x)=\alpha$. In particular, $(X_j,x)$ is a smooth surface germ. The divisor $\pi_j^{*}(0)$ is reduced, the fiber $\pi_j^{-1}(0)$ is smooth, and each of its components is isomorphic to $\A^1$. }

\smallskip 

\noindent \emph{Proof of the claim.} Replacing $u$ by $u-\alpha$ one may assume that $\alpha=0$. All the assertions but the last one are local, hence one may restrict to the neighborhood $U_j:=U_j(0)$ of $x$. 

We proceed by induction on $j$. Since $p_1(u)/u\in\mathbb{k}[u]$ does not vanish in $U_0$ one has by  (\ref{eq: bush}):
$$\cO_{U_1}(U_1)=\cO_{U_0}(U_0)[p_1(u)/z]=\cO_{U_0}(U_0)[u/z]\,.$$ 
Thus, $U_1\to U_0$ is an affine modification of $U_0\subset\A^2=\Spec \mathbb{k}[z,u]$ along the reduced divisor $z^*(0)$ whose center is  the ideal $(z,u)$. This yields an embedding $U_1\hookrightarrow\A^2=\Spec \mathbb{k}[z,t'_1]$ where $t'_1=u/z$ commuting with the projections to $\A^1=\Spec \mathbb{k}[z]$ and sending the exceptional divisor to the axis $\{z=0\}$ in $\A^2$. Since $t_1=p_1(u)/z$ and $t_1'$ differ  in $U_1$ by the invertible factor $p_1(u)/u$ our assertions hold for $j=1$.

Assume by recursion that for some $j<l$ there is an embedding $U_j\hookrightarrow \A^2=\Spec \mathbb{k}[z,t_j']$ which commutes with the projections to $\A^1=\Spec \mathbb{k}[z]$. Consider the map $U_{j+1}\to U_{j}$ forgetting the last coordinate $t_{j+1}=r_{j+1}(u)t_{j}/z$, see \eqref{eq: bush}. Since $j+1\le l$ the function $r_{j+1}$ is invertible in $U_{j}$, see \eqref{eq: order-1}. 
Hence
\be\la{eq:ind-step} \cO_{U_{j+1}}(U_{j+1})=\cO_{U_{j}}(U_{j})[r_{j+1}(u)t_{j}/z]=\cO_{U_{j}}(U_{j})[t_{j}/z]\,.\ee This yields an embedding $U_{j+1}\hookrightarrow\A^2=\Spec \mathbb{k}[z,t_{j+1}']$ commuting with the projections to $\A^1=\Spec \mathbb{k}[z]$ where $t_{j+1}':=t_{j}/z$ differs  in $U_{j+1}$ from $t_{j+1}$ by an invertible factor $r_{j+1}$. Hence $(z,t_{j+1})$ gives local analytic coordinates near the axis $z=0$. This proves the claim for $j=1,\ldots,l$.

The first equation in \eqref{eq: bush}
gives $u/z=ut_1/p_1(u)\in\cO_{U_j}(U_j)$. Hence $r_{j+1}(u)/z\in\cO_{U_j}(U_j)$, and so, $ \cO_{U_{j+1}}(U_{j+1})=\cO_{U_{j}}(U_{j})$ for all $j\in\{l,\ldots,m-1\}$, see \eqref{eq:ind-step}. This means that $U_{j+1}\to U_{j}$ is an isomorphism commuting with the projections to $\A^1=\Spec \mathbb{k}[z]$. Now the claim follows.

\smallskip 

According to the claim for any $j=1,\ldots,m$ the surface $X_j$ is smooth, the fiber $\pi_j^*(0)$ is reduced, and the unique component  $F_{\alpha,j}\subset X_j$ of this fiber with $u|_{F_{\alpha,j}}=\alpha$ is an $\A^1$-curve parameterized by $t_l$
where $l=\min\{j,l(\alpha)\}$. This yields (ii). 

It follows also that the morphism $\rho_{j+1}\colon X_{j+1}\to X_j$ as in (iii) sends the affine line $F_{\alpha, j+1}\subset X_{j+1}$ isomorphically onto the affine line $F_{\alpha, j}\subset X_{j}$  for any  $j=l,\ldots,m-1$ and contracts $F_{\alpha, j+1}$ to the point $(0, \alpha,0,\ldots,0)\in\A^{j+2}$   for any $j=0,\ldots, l-1$. This proves (iii) and shows as well that $F_{\alpha, j}$ is on level $j$ for $1\le j\le l-1$ and on level $l$ for $l\le j\le m$. Hence one has $\Gamma_0(\pi_j)= (\Gamma_0)_{\le j}$ as stated in (iv). 
\eproof

\brem\la{rem:localization}
It follows from the proof that $\Gamma_0(\pi_j|_{U_j(\alpha)})=(\cB_\alpha)_{\le j}$ for $j=1,\ldots,l(\alpha)$ and $\Gamma_0(\pi_j|_{U_j(\alpha)})=\cB_\alpha$ for $j=l(\alpha),\ldots,m$. In particular,
the graph $\Gamma_0(\pi_m|_{U_m(\alpha)})$ coincides with the branch $\cB_\alpha$  of $\Gamma_0$ of height $\h(\cB_\alpha)=l(\alpha)$. 
\erem

From Proposition \ref{9.1} and its proof we deduce such a corollary.

\bcor\la{cor:loc-coord} We keep the notation as before. Given a root  $\alpha$ of $p_1$ consider the fiber component $F_\alpha=F_{\alpha,m}=\{z=0, u-\alpha=0\}\subset\pi_m^{-1}(0)$ on $X_m$. Then  the following hold. 
 \begin{enumerate}
 \item[{\rm (a)}] $t_{l(\alpha)}$ coincides with $\frac{u-\alpha}{z^{l(\alpha)}}$ up to a factor which is a regular and invertible function in a Zariski open neighborhood $U_m(\alpha)\subset X_m$ of $F_\alpha$; 
\item[{\rm (b)}] $(z, t_{l(\alpha)})$ yields a local analytic coordinate system in $X_m$ near $F_\alpha$; 
\item[{\rm (c)}] for $i\in\{1,\ldots,m\}$ the restriction $r_i(u)t_i|_{F_\alpha}$ vanishes if $i\neq l(\alpha)$ and gives an affine coordinate on $F_\alpha\cong\A^1$ if $i=l(\alpha)$. \end{enumerate}
\ecor

\bproof  Statements (a)  and (b) follow by an argument in the proof of Proposition \ref{9.1}
(see the Claim). 

(c) By \eqref{eq: order-1} one has $r_i(\alpha)=0$ for $i>l(\alpha)$ and $r_i(\alpha)\neq 0$ for $i\le l(\alpha)$. Due to (b),  $r_i(\alpha)t_i|_{F_\alpha}$ is an affine coordinate if $i=l(\alpha)$. If $i<l(\alpha)$ then $r_{i+1}(\alpha)\neq 0$ and, by \eqref{eq: bush}, $t_i|_{F_\alpha}=z\frac{t_{i+1}}{r_{i+1}(\alpha)}|_{F_\alpha}=0$.
\eproof

\section{Spring bushes  versus bushes} \label{sec:spring-bushes}
In this section we extend the results of Section \ref{sec:bushes} to the Danielewski-Fieseler surfaces whose fiber trees are \emph{spring bushes}. 
 
\bdefi[\emph{Spring bush}]
\label{def: spring-bush}
A rooted tree $\widehat\Gamma$ of height $\hat h\ge 1$ is called a \emph{spring bush} if the truncation
$\Gamma:=\widehat\Gamma_{\le\hat h-1}$ is  a bush  
sharing  the same root with $\Gamma$. Thus, any leaf of $\widehat\Gamma$ is a neighbor of a leaf of $\Gamma$. A spring bush is not necessarily a bush  (see Figure \ref{fig:spring-bush}).
\edefi

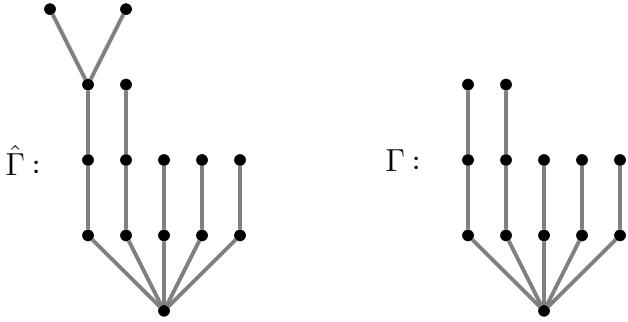
\begin{figure}[h]
\begin{tabular}{l@{\hspace{70pt}}l@{\hspace{70pt}}l}
    \pgfdeclarelayer{background}
		\pgfsetlayers{background,main}
    \begin{tikzpicture}[scale=2,>=latex]
      \node[vertex] (a) at (0.5,0) {};
      \node[vertex] (b1) at (0,0.5) {};
      \node[vertex] (b2) at (0.25,0.5) {};
      \node[vertex] (b3) at (0.5,0.5) {};
      \node[vertex] (b4) at (0.75,0.5) {};
      \node[vertex] (b5) at (1,0.5) {};
       \node[vertex] (c1) at (0,1) {};
       \node[vertex] (c2) at (0.25,1) {};
       \node[vertex] (c3) at (0.5,1) {};
       \node[vertex] (c4) at (0.75,1) {};
       \node[vertex] (c5) at (1,1) {};
       \node[vertex] (d1) at (0,1.5) {};
       \node[vertex] (d2) at (0.25,1.5) {};
       \node[vertex] (e1) at (-0.25,2) {};
       \node[vertex] (e2) at (0.25,2) {};
       \node[left] at (c1.west) {$\hat\Gamma:\quad$} {};
      
      \draw[edge] (a)edge(b1) (a)edge(b2) (a)edge(b3) (a)edge(b4) (a)edge(b5) (b1)edge(c1) (b2)edge(c2) (b3)edge(c3) (b4)edge(c4) (b5)edge(c5) (c1)edge(d1)  (c2)edge(d2) (d1)edge(e1)  (d1)edge(e2);
    
  \begin{scope}[xshift=2.5cm]
      \node[vertex] (a1) at (0.5,0) {};
      \node[vertex] (b11) at (0,0.5) {};
      \node[vertex] (b21) at (0.25,0.5) {};
      \node[vertex] (b31) at (0.5,0.5) {};
      \node[vertex] (b41) at (0.75,0.5) {};
      \node[vertex] (b51) at (1,0.5) {};
       \node[vertex] (c11) at (0,1) {};
       \node[vertex] (c21) at (0.25,1) {};
       \node[vertex] (c31) at (0.5,1) {};
       \node[vertex] (c41) at (0.75,1) {};
       \node[vertex] (c51) at (1,1) {};
       \node[vertex] (d11) at (0,1.5) {};
       \node[vertex] (d21) at (0.25,1.5) {};
       \node[left] at (c11.west) {$\Gamma:\quad$} {};
      
      \draw[edge] (a1)edge(b11) (a1)edge(b21) (a1)edge(b31) (a1)edge(b41) (a1)edge(b51) (b11)edge(c11) (b21)edge(c21)  (b31)edge(c31) (b41)edge(c41) (b51)edge(c51) (c11)edge(d11)  (c21)edge(d21) ;
\end{scope}
         \end{tikzpicture}
\end{tabular}
  \caption{A spring bush $\hat\Gamma$  of type ${\rm tp}(\hat\Gamma)=(0,0,3,1,2)$ over a bush $\Gamma=\hat\Gamma_{\le 3}$
}
  \la{fig:spring-bush}
\end{figure}

\bdefi[\emph{Accompanying sequence of a spring bush}]\la{def: acomp-system-of-spring-bush} Let $\hat\Gamma$  be a spring bush of height $\hat h=\h(\hat\Gamma)$ over a bush
$\Gamma$ of height $h=\hat h-1$, and let $(p_i,r_i)_{i=1,\ldots,h}$ be the accompanying sequence of $\Gamma$. Recall that for each root $\alpha$ of $p_1$ there corresponds a unique linear branch $\cB_\alpha$ of $\Gamma$ (see Definition \ref{def: adopted}). Let $p_{\hat h}$ be the monic polynomial  dividing  $p_{h}$ such that the roots of $p_{\hat h}$ correspond to the branches of $\Gamma$ linked to the top vertices of $\hat\Gamma$. Let also $r_{\hat h}=p_1/p_{\hat h}$. For instance, for $\hat\Gamma$ in Figure \ref{fig:spring-bush} along with the accompanying sequence \eqref{eq:accomp-sec} of $\Gamma$ and the root $\alpha_1=0$ of $p_1$ one has $(p_4,r_4)=(u,p_1/u)$. 

For a branch $\cB_\alpha$ of $\Gamma$ linked  to $n(\alpha)$ leaves of $\hat\Gamma$ we fix a monic polynomial  with simple roots $q_\alpha\in \mathbb{k}[v]$  of degree $n(\alpha)$ where $q_\alpha(v)=v$ if $n(\alpha)=1$.  We let 
\be\la{eq: q} q(u,v)=\sum_{p_1(\alpha)=0} q_\alpha(v)\frac{p_1(u)}{u-\alpha}\in \mathbb{k}[u,v]\,.\ee
The system of polynomials $\{(p_i,r_i)_{i=1,\ldots,\hat h}, q\}$ is called an \emph{accompanying sequence of $\hat\Gamma$}. 
\edefi

Our main result in this section is the following.

\bprop\la{9.2} Let $\hat \Gamma$ be a spring bush, and let $\tilde\Gamma$ be a bush with  $\tp(\widehat \Gamma)=\tp(\widetilde\Gamma)$, see  Figure {\rm \ref{fig:spring-bushe-versus-bush}}.  Letting $B=\Spec\mathbb{k}[z]\cong\A^1$
consider Danielewski-Fieseler surfaces $\pi_{\hat X}\colon\hat X\to B$ and $\pi_Y\colon Y\to B$ with the unique special fibers over $0\in B$ such that $\Gamma_0(\hat X)=\hat \Gamma$ and $\Gamma_0(Y)=\tilde\Gamma$. Then there is an isomorphism of cylinders  $\phi\colon\hat \cX\stackrel{\cong_B}{\longrightarrow}\cY$ which sends the components of $z^{*}(0)$ in $\hat \cX$ to components of $z^{*}(0)$ in $\cY$ preserving the levels. 
\eprop

\begin{figure}[h]
\begin{tabular}{l@{\hspace{70pt}}l@{\hspace{70pt}}l}
    \pgfdeclarelayer{background}
		\pgfsetlayers{background,main}
    \begin{tikzpicture}[scale=2,>=latex]
      \node[vertex] (a) at (0.5,0) {};
      \node[vertex] (b1) at (0.25,0.5) {};
      \node[vertex] (b3) at (0.75,0.5) {};
       \node[vertex] (c1) at (0,1) {};
       \node[vertex] (c2) at (0.5,1) {};
       \node[left] at (b1.west) {$\hat\Gamma:\quad$} {};
      
      \draw[edge] (a)edge(b1) 
                            (a)edge(b3) 
                           (b1)edge(c1) (b1)edge(c2); 

\begin{scope}[xshift=2.5cm]
      \node[vertex] (a2) at (0.5,0) {};
      \node[vertex] (b12) at (0.25,0.5) {};
      \node[vertex] (b22) at (0.5,0.5) {};
      \node[vertex] (b32) at (0.75,0.5) {};
       \node[vertex] (c12) at (0.25,1) {};
       \node[vertex] (c22) at (0.5,1) {};
       \node[left] at (b12.west) {$\tilde\Gamma:\quad$} {};
      
      \draw[edge] (a2)edge(b12) (a2)edge(b22) (a2)edge(b32)
(b12)edge(c12) (b22)edge(c22) ; 
\end{scope}
         \end{tikzpicture}
\end{tabular}
\caption{A spring bush $\hat\Gamma$  and a bush
 $\tilde\Gamma$ of the same type $(0,1,2)$
}
  \la{fig:spring-bushe-versus-bush}
\end{figure}
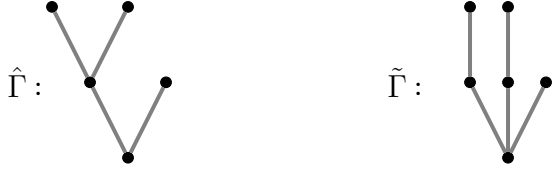

The proof is done at the end of this section. We start with the following elementary fact that can be interpreted in terms of commuting affine modifications, cf.\ I.1.4.2.

\blem\label{new2} Let $A$ be an affine domain, and let $a_1, \ldots , a_k, b_1, \ldots , b_l\in{\rm Frac}(A)$ be elements of the field of fractions of $A$. Consider the extensions
$A'=A[a_1, \ldots , a_k]$, $B=A[b_1, \ldots , b_l]$, and $B' =A[a_1, \ldots , a_k, b_1, \ldots , b_l]$. Then $B'=B[a_1, \ldots , a_k]$, that is,  the extensions $A\subset A'$ and $B\subset  B'$ share the same system of generators $a_1, \ldots , a_k$.
\elem

\bnota\la{not:three-bushes} In the remaining part of this section we consider the following objects.
\begin{itemize}\item $\hat\Gamma$ - a spring bush of height $\h(\hat\Gamma)=m+1\ge 3$ over a bush $\Gamma$ along with an accompanying sequence $\{(p_i,r_i)_{i=1,\ldots,m+1}, q\}$ as in \ref{def: acomp-system-of-spring-bush}. According to Remark \ref{rem: 9.1a} we may and we will assume that  $\Gamma$ has no branch of height $1$;

\item $\Gamma_h$ - the subbush of height $m+1$ of $\hat\Gamma$ obtained by removing on each branch  of $\hat\Gamma$ of height $m+1$ all leaves but one, see Figure \ref{fig:spring-bushes};

\item $\Gamma_s$ - the subbush of height $m$ of $\Gamma_h$ obtained from $\Gamma_h$ by removing all the leaves of $\Gamma_h$, see Figure \ref{fig:spring-bushes};

\item 
$\cB_s(\alpha)$ - the branch of
$\Gamma_s$ 
of height $l_s(\alpha)=\h(\cB_s (\alpha))$ which corresponds to a root $\alpha$ of $p_1$;
$\cB_h(\alpha)\subset\Gamma_h$  and $l_h(\alpha)$ are defined likewise;

\item $B=\Spec \mathbb{k}[z]\cong\A^1$;

\item $\hat\pi\colon\hX\to B$,  $\pi_h\colon X_h\to B$, and $\pi_s\colon X_s\to B$ - the Danielewski-Fieseler surfaces with a unique special fiber over $0\in B$ such that $\Gamma_0(\hat\pi)=\hat\Gamma$, $\Gamma_0(\pi_h)=\Gamma_h$, and $\Gamma_0(\pi_s)=\Gamma_s$ where

\item $X_h\subset \A^{m+3}$ and $X_s\subset \A^{m+2}$ satisfy the corresponding systems \eqref{eq: bush}. So, \eqref{eq: bush} provides the extensions
 $$\mathbb{k}[z,u]\subset\cO_{X_s}(X_s)\subset\cO_{X_h}(X_h)\,;$$

\item  
$F_s(\alpha)\subset X_s$ - the fiber component of $z^*(0)$ corresponding to the leaf of 
$\cB_s (\alpha)$;

\item 
$t_s(\alpha)=(u-\alpha)/z^{l_s(\alpha)}$ - an affine coordinate on 
$F_s(\alpha)$, 
see Corollary \ref{cor:loc-coord}(a);

\item $F_h(\alpha)\subset X_h$ and $t_h(\alpha)$ are defined likewise;

\item $\hat\cX, \cX_h, \cX_s$ - the cylinders over $\hX, X_h$, and $X_s$, respectively;
\item $\cF_s(\alpha)=F_s(\alpha)\times\A^1\cong\Spec \mathbb{k}[t_s(\alpha),v]\cong\A^2$ - 
the fiber component of $z^*(0)$ in $\cX_s$ which corresponds to 
$\alpha$;  $\cF_h(\alpha)\subset \cX_h$ is  defined likewise.

\end{itemize}
\enota

\begin{figure}[h]
\begin{tabular}{l@{\hspace{70pt}}l@{\hspace{70pt}}l}
    \pgfdeclarelayer{background}
		\pgfsetlayers{background,main}
    \begin{tikzpicture}[scale=2,>=latex]
      \node[vertex] (a) at (0.5,0) {};
      \node[vertex] (b1) at (0.25,0.5) {};
      \node[vertex] (b2) at (0.5,0.5) {};
      \node[vertex] (b3) at (0.75,0.5) {};
       \node[vertex] (c1) at (0.25,1) {};
       \node[vertex] (c2) at (0.5,1) {};
       \node[vertex] (c3) at (0.75,1) {};
       \node[vertex] (d1) at (0.25,1.5) {};
       \node[vertex] (d2) at (0.5,1.5) {};
       \node[vertex] (e1) at (0,2) {};
       \node[vertex] (e2) at (0.5,2) {};
       \node[left] at (c1.west) {$\hat\Gamma:\quad$} {};
      
      \draw[edge] (a)edge(b1) (a)edge(b2) (a)edge(b3) 
(b1)edge(c1) (b2)edge(c2) (b3)edge(c3) 
(c1)edge(d1)  (c2)edge(d2) (d1)edge(e1)  (d1)edge(e2);
    
  \begin{scope}[xshift=2.5cm]
      \node[vertex] (a1) at (0.5,0) {};
      \node[vertex] (b11) at (0.25,0.5) {};
      \node[vertex] (b21) at (0.5,0.5) {};
      \node[vertex] (b31) at (0.75,0.5) {};
       \node[vertex] (c11) at (0.25,1) {};
       \node[vertex] (c21) at (0.5,1) {};
       \node[vertex] (c31) at (0.75,1) {};
       \node[vertex] (d11) at (0.25,1.5) {};
       \node[vertex] (d21) at (0.5,1.5) {};
       \node[vertex] (e11) at (0.25,2) {};
       \node[left] at (c11.west) {$\Gamma_h:\quad$} {};
      
      \draw[edge] (a1)edge(b11) (a1)edge(b21) (a1)edge(b31) 
(b11)edge(c11) (b21)edge(c21)  (b31)edge(c31) 
(c11)edge(d11)  (c21)edge(d21)
(d11)edge(e11) ; 
\end{scope}

\begin{scope}[xshift=5cm]
      \node[vertex] (a2) at (0.5,0) {};
      \node[vertex] (b12) at (0.25,0.5) {};
      \node[vertex] (b22) at (0.5,0.5) {};
      \node[vertex] (b32) at (0.75,0.5) {};
       \node[vertex] (c12) at (0.25,1) {};
       \node[vertex] (c22) at (0.5,1) {};
       \node[vertex] (d12) at (0.25,1.5) {};
       \node[left] at (c12.west) {$\Gamma_s:\quad$} {};
      
      \draw[edge] (a2)edge(b12) (a2)edge(b22) (a2)edge(b32) 
(b12)edge(c12) (b22)edge(c22) (c12)edge(d12) ;
\end{scope}
         \end{tikzpicture}
\end{tabular}
\caption{A spring bush $\hat\Gamma$  
and its subbushes $\Gamma_h$ and $\Gamma_s$
}
  \la{fig:spring-bushes}
\end{figure}
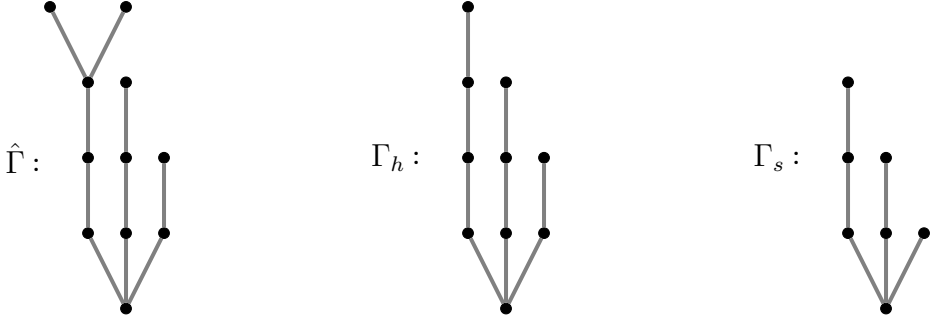

\blem\label{new1} There is an isomorphism $\hat \cX\cong_B V$ where the affine threefold $V$ results from  the affine modification $V\to\cX_h$ along the divisor $z^*(0)$ on $\cX_h$ with center  the ideal $(z,q(u,v))\subset\cO_{\cX_h}(\cX_h)$ for $q\in \mathbb{k}[u,v]$ as in \eqref{eq: q}.
\elem

\bproof Let $t_s\in\cO_{X_s}(X_s)$ be such that $t_s|_{F_s(\alpha)}=t_s(\alpha)$ for any root $\alpha$ of $p_1$. The extension
\be\la{eq:ext-1} \cO_{X_s}(X_s)\subset\cO_{X_h}(X_h)=\cO_{X_s}(X_s)[t_s/z]\,\ee  amounts to the fiber modification $\sigma_h\colon X_h\to X_s$ along the divisor $z^*(0)$ whose center is  the ideal $(z, t_s)\subset\cO_{X_s}(X_s)$. The extension
\be\la{eq:ext-2} \cO_{X_s}(X_s)\subset\cO_{\hX}(\hX)=\cO_{X_s}(X_s)[q(u,t_s)/z]\,\ee amounts to the fiber modification $\hX\to X_s$ along the divisor $z^*(0)$ whose center is  the ideal $(z, q(u,t_s))\subset\cO_{X_s}(X_s)$. 

Letting $\cO_{\cX_s}(\cX_s)=\cO_{X_s}(X_s)[v]$, by Lemma I.5.1 applied to \eqref{eq:ext-2} the cylinder $\hat \cX$ can be obtained from 
$\cX_s$ via the affine modification $\hat\sigma\colon\hat \cX\to \cX_s$ along the divisor $z^*(0)$ whose center is  the ideal
$J_s=(z,q(u,t_s), v)\subset\cO_{\cX_s}(\cX_s)$, that is,
\be\la{eq:ext-3} \cO_{\hat \cX}(\hat \cX)\cong_{\mathbb{k}[z]}\cO_{\cX_s}(\cX_s)[q(u,t_s)/z,v/z]\,.\ee 
The ideal $J_s$ defines a reduced zero dimensional subscheme $\cS\subset\cX_s$. Let 
$$\cS(\alpha)=\cS\cap\cF_s(\alpha)=\cF_s(\alpha)\cap\{q_\alpha(t_s)=0\}\,.$$ 
One has ${\rm card}\,\cS(\alpha)=n(\alpha)\ge 1$, see \eqref{eq: q}. In the case $l_s(\alpha)<m$ one has $n(\alpha)=1$ and $\cS_\alpha=\{t_s=0,v=0\}\subset\cF_s(\alpha)$.
Let $$\mathfrak{S}(\alpha)=\{t_s=0,q_\alpha(v)=0\}\subset\cF_s(\alpha)\,.$$ One has
${\rm card}\,\cS(\alpha)={\rm card}\,\mathfrak{S}(\alpha)$ and, moreover, $\cS(\alpha)=\mathfrak{S}(\alpha)$ if $l_s(\alpha)<m$, see \ref{def: acomp-system-of-spring-bush}. According to Theorem  I.4.4 the relative flexibility holds for the cylinder $\cX_s$  (see Definition I.4.2). Hence there is an automorphism $\tau\in\SAut_B(\cX_s)$ such that 
\begin{itemize} \item $\tau(\cF_s(\alpha))=\cF_s(\alpha)$ $\forall\alpha$;
\item $\tau|_{\cF_s(\alpha)}={\rm id}$ if $l_s(\alpha)<m$; 
\item
$\tau(\cS(\alpha))= \mathfrak{S}(\alpha)$ if $l_s(\alpha)=m$. 
\end{itemize}
Thus, $\tau(\cS(\alpha))= \mathfrak{S}(\alpha)$ for any $\alpha$, and so, $\tau_*$ sends $J_s$ to the ideal $I_s:=(z, t_s, q(u,v))\subset\cO_{\cX_s}(\cX_s)$.

Consider further the affine threefold $W=\Spec \cO_{\cX_s}(\cX_s)[t_s/z, q(u,v)/z]$. It results from the affine modification of $\cX_s$ along the divisor $z^*(0)$ whose center is the ideal $I_s$. By virtue of \eqref{eq:ext-3}  and Lemma I.1.5, $\tau$ admits a lift to an isomorphism
$\hat \cX\stackrel{\cong_B}{\longrightarrow} W$. We claim that there is an isomorphism $W\cong_B V$. Indeed,
using \eqref{eq:ext-1} and Lemma \ref{new1} one obtains 
$$ 
 \cO_{W}(W)  =  \cO_{\cX_s}(\cX_s)[t_s/z, q(u,v)/z] $$
$$ = \cO_{\cX_s}(\cX_s)[t_s/z][q(u,v)/z] \cong_{\mathbb{k}[z]}  \cO_{\cX_h}(\cX_h)[q(u,v)/z]  = \cO_{V}(V)  \,.
$$
\eproof

In the  proof of Proposition \ref{9.2} we use the following Lemmas \ref{lem:triv-seq-X-h}--\ref{lem:triv-seq-Y}.

\blem\la{lem:triv-seq-X-h} We 
adopt Notation {\rm \ref{not:three-bushes}}. 
To a trivializing sequence \eqref{eq: trivializing-seq-aff-modif} of $X_h$:
\be\la{eq: triv-seq-X-h} 
X_h =  X_{h,m+1} \stackrel{\varrho_{m+1}}{\longrightarrow} X_{h,m} \stackrel{}{\longrightarrow}  
\ldots \stackrel{}{\longrightarrow}   X_{h,1} \stackrel{\varrho_{1}}{\longrightarrow} 
X_{h,0}=B\times\A^1\,
\ee 
there corresponds a trivializing sequence of cylinders
\be\la{eq: triv-seq-cX-h} 
\cX_h =  \cX_{h,m+1} \stackrel{r_{m+1}}{\longrightarrow} \cX_{h,m} \stackrel{}{\longrightarrow}  
\ldots \stackrel{}{\longrightarrow}   \cX_{h,1} \stackrel{r_{1}}{\longrightarrow} 
\cX_{h,0}=B\times\A^2\,
\ee
where $r_{i+1}=\varrho_{i+1}\times {\rm id_{\A^1}}$, $i=0,\ldots,m$. Given a root $\alpha$ of $p_1$ let $F_{h,i}(\alpha)$ be the corresponding component  of the divisor $z^{*}(0)$ on $X_{h,i}$  equipped with the affine coordinate $$t_{h,i}(\alpha)=(u-\alpha)\cdot z^{-l_{h,i}(\alpha)}$$ where $(u-\alpha)|_{F_{h,i}(\alpha)}=0$ and $l_{h,i}(\alpha)$ is the level of $F_{h,i}(\alpha)$,
see Corollary {\rm \ref{cor:loc-coord}(a)}. Then
the affine modification $r_{i+1}\colon  \cX_{h,i+1}\to
\cX_{h,i}$ amounts to the extension 
\be\la{eq:ext-X-h}\cO_{\cX_{h,i}}(\cX_{h,i})\subset \cO_{\cX_{h,i+1}}(\cX_{h,i+1})=\cO_{\cX_{h,i}}(\cX_{h,i})[t_{h,i}/z]\,\ee where $t_{h,i}\in\cO_{X_{h,i}}(X_{h,i})$ satisfies $$t_{h,i}|_{F_{h,i}(\alpha)}=\begin{cases} t_{h,i}(\alpha), &  l_h(\alpha)\ge i+1, \\ 0, & l_h(\alpha)\le i\,.\end{cases}$$
\elem

\bproof
The center of the affine modification $r_{i+1}\colon \cX_{h,i+1}\to\cX_{h,i}$ along the divisor $z^*(0)$ on $\cX_{h,i}$ is the union of the affine lines 
\be\la{eq:lines} L_{h,i}(\alpha)=\{t_{h,i}(\alpha)=0\}\cong\Spec \mathbb{k}[v]\cong\A^1\,\ee such that $l_h(\alpha)\ge i+1$. 
Each component 
\be\la{eq:planes} \cF_{h,i}(\alpha)=F_{h,i}\times\A^1\cong\Spec\mathbb{k}[t_{h,i}(\alpha),v]\cong\A^2\ee of $z^*(0)$ in $\cX_{h,i}$ contains at most one such line. Now the lemma follows easily.
\eproof

\blem\la{lem:V} Let $\tp(\hat\Gamma)=(n_0,\ldots,n_{m+1})$. Consider the polynomials 
$$ q_\alpha\in\mathbb{k}[u],\quad\deg q_\alpha=n(\alpha), \quad\mbox{and}\quad q\in \mathbb{k}[u,v]$$ 
as in \eqref{eq: q}, see {\rm \ref{def: acomp-system-of-spring-bush}}. 
For each $i=1,\ldots,m+1$  consider  
the affine threefold $$V_i=\Spec\cO_{\cX_{h,i}}(\cX_{h,i})[q(u,v)/z]$$ given  in $\cX_{h,i}\times\A^1$ by equation $zw-q(u,v)=0$ where $\A^1=\Spec \mathbb{k}[w]$. Then the natural projection $V_i\to\cX_{h,i}$ is the affine modification along the divisor $z^*(0)$ on $\cX_{h,i}$ with the reduced center $\cL_i=\cup_{\alpha,\beta} L_{\alpha,\beta}$ where  the affine line$$L_{\alpha,\beta}=\{v-\beta=0\}\subset\cF_{h,i}(\alpha)\cong\A^2\,,$$ see \eqref{eq:planes}, corresponds
to a solution $(\alpha,\beta)\in\A^2=\Spec \mathbb{k}[u,v]$ of  the system
$$p_1(u)=0,\quad q(u,v)=0\,.$$  The plane $ \cF_{h,i}(\alpha)$ contains $n(\alpha)$ such lines $L_{\alpha,\beta}$. 
The center of the modification $V_i\to\cX_{h,i}$ consists of 
\be\la{eq:N} N=\sum_{p_1(\alpha)=0} n(\alpha)=\sum_{k=0}^{m+1} n_k\ee affine lines. The reduced divisor $z^*(0)$ on $V_i$ has $N$ disjoint components $\cF_{i,(\alpha,\beta)}$ where
\be\la{eq:new-planes} \cF_{i,(\alpha,\beta)}=\{z=0,\, u=\alpha,\,v=\beta\}\cong\Spec \mathbb{k}[t_{h,i}(\alpha),w]\cong\A^2\,.\ee
\elem

\bproof
The proof is straightforward. 
\eproof

\blem\la{lem:seq-V} There is a sequence 
\be\la{eq: triv-seq-V} V_{m+1}\stackrel{\nu_{m+1}}{\longrightarrow} V_m\longrightarrow\ldots{\longrightarrow} V_2\stackrel{\nu_{2}}{\longrightarrow} V_1\,\ee where 
$\nu_{i+1}\colon V_{i+1}\to V_i$ is the affine modification along the divisor $z^*(0)$ with a reduced center $\mathfrak{S}_i$ defined by the ideal $(z,t_{h,i})\subset\cO_{V_{i}}(V_{i})$ and consisting of the 
\be\la{eq:N-i} N_i=\sum_{l_h(\alpha)\ge i+1} n(\alpha)=\sum_{k=i+1}^{m+1} n_k \ee affine lines
$L_{i,(\alpha,\beta)}$ with $l_h(\alpha)\ge i+1$. 
\end{lem}

\bproof
 Using Lemma \ref{new2} and \eqref{eq:ext-X-h} one obtains 
$$\cO_{V_{i+1}}(V_{i+1}) = \cO_{\cX_{h,i+1}}(\cX_{h,i+1}) [q(u,v)/z] = 
\cO_{\cX_{h,i}}(\cX_{h,i})[t_{h,i}/z][q(u,v)/z]$$
$$=\cO_{\cX_{h,i}}(\cX_{h,i})[q(u,v)/z][t_{h,i}/z]=\cO_{V_{i}}(V_{i})[t_{h,i}/z]\,.$$
\eproof

\blem\la{lem:triv-seq-Y} 
Let $\pi_Y\colon Y\to B=\A^1$ be a Danielewski-Fieseler surface as in  Proposition {\rm \ref{9.2}}, that is, 
$\Gamma_0(Y)$ is the bush $\tilde\Gamma$ with $\tp(\tilde\Gamma) = \tp(\hat\Gamma) = (n_0,n_1,\ldots,n_{m+1})$. Given a trivializing sequence
\be\la{eq: triv-seq-Y} Y=Y_{m+1}\stackrel{\sigma_{m+1}}{\longrightarrow} Y_m\longrightarrow\ldots\longrightarrow Y_1\stackrel{\sigma_{1}}{\longrightarrow} Y_0=B\times\A^1=\Spec \mathbb{k}[z,u]\,\ee
with 
$\Gamma_0(\pi_{Y_i})=\tilde\Gamma_{\le i}$ there is a trivializing sequence of cylinders  
\be\la{eq: triv-seq-cY} \cY=\cY_{m+1}\stackrel{s_{m+1}}{\longrightarrow} \cY_m\longrightarrow\ldots\longrightarrow \cY_1\stackrel{s_{1}}{\longrightarrow} \cY_0=B\times\A^2\,\ee where  for $i\ge 1$
 the morphism $\colon  \cY_{i+1}\to\cY_i$, $s_{i+1}=\sigma_{i+1}\times {\rm id_{\A^1}}$ is an affine modification along the reduced divisor $z^*(0)$ on $\cY_i$ with a reduced center $\cS_i$ such that 
 \begin{itemize}\item $z^{-1}(0)$ is a disjoint union of $N$ components $\tilde\cF_{i,j}\cong\A^2$ with $N$ as in \eqref{eq:N} where $\tilde\cF_{i,j}$ corresponds to the vertex on level $l(\tilde\cF_{i,j})$ of the branch $\tilde\cB_j$ of $\tilde\Gamma$;
 \item
the center $\cS_i$ is a disjoint union of $N_i$ affine lines $\tilde L_{i,j}\subset\tilde \cF_{i,j}$ with $\h(\tilde\cB_j)\ge i+1$ where $N_i$ is as in \eqref{eq:N-i} and $\cS_i\cap \cF_{i,j}$ consists of at most one line $L_{i,j}$.  
\end{itemize}
\elem

\bproof
Since $\tp(\tilde\Gamma) = \tp(\hat\Gamma) $ using  \eqref{eq:N-i} one obtains
\be\la{eq:N-i-prime}\card\{j\,\vert\,\h(\tilde\cB_j)\ge i+1\}=\sum_{k=i+1}^{m+1} n_k=N_i\,.\ee The remaining assertions are immediate. 
\eproof

\medskip

\noindent \emph{Proof of Proposition} \ref{9.2}.  By virtue of Theorem I.5.7 one may assume that 
$X_h$ satisfies \eqref{eq: bush} in  Proposition \ref{9.1} with $j=m+1=\h(\Gamma_h)=\h(\hat \Gamma)$. Due to Remark \ref{rem: 9.1a} one may suppose as well  that  $\Gamma_h$ do not have branches of height 1. By virtue of Lemma \ref{new1} the assertion of Proposition \ref{9.1} follows from the next claim by 
letting $i=m+1$.

\smallskip

\noindent {\bf Claim.} \emph{For any $i=1,\ldots,m+1$ there is an isomorphism $V_i\cong_B \cY_i$ which sends the components of $z^*(0)$ in $V_i$ to components of $z^*(0)$ in $\cY_i$ preserving the levels.} 

\smallskip

\noindent\emph{Proof of the claim.} We proceed by induction on $i$ using the trivializing sequences \eqref{eq: triv-seq-V} and  \eqref{eq: triv-seq-cY}.
By our assumption  for $i=1$ one has $N_1=N_0=N$, see \eqref{eq:N-i-prime}. Recall (see Lemma \ref{lem:V}) that $V_1\subset\A^5=\Spec \mathbb{k}[z,u,v,t,w]$ is given by  $$zt-p_1(u)=zw-q(u,v)=0\,.$$
Thus, $V_1$ results from an affine modification $V_1\to\A^3=\Spec \mathbb{k}[z,u,v]$ along the divisor $z^*(0)$  whose center is the ideal $(z, p_1(u), q(u,v))\subset \mathbb{k}[z,u,v]$ supported by a reduced zero dimensional scheme $\mathfrak{S}_0\subset\{z=0\}$ of cardinality $N$. By Lemma I.5.1 the cylinder $\cY_1$ results as well from an affine modification $\cY_1\to\A^3=B\times\A^2=\Spec \mathbb{k}[z,u,v]$ along the divisor $z^*(0)$ whose center is a reduced zero dimensional scheme $\cS_0\subset\{z=0\}$ of cardinality $N$. 
There is an automorphism $\tau$ of $\A^3$ with $\tau_*(z)=z$ and $\tau(\mathfrak{S}_0)=\cS_0$. By Lemma
I.1.5, $\tau$ can be lifted to an isomorphism $\varphi_1\colon V_1 \stackrel{\cong_B}{\longrightarrow} \cY_1$.
Thus, the claim holds for $i=1$. 

Suppose further that for some $i\in\{1,\ldots,m\}$ there is an isomorphism $\varphi_i\colon V_i\stackrel{\cong_B}{\longrightarrow}\cY_i$ which transforms
the $N$ affine planes $\cF_{i,(\alpha,\beta)}$ in \eqref{eq:new-planes} to the $N$ affine planes $\tilde\cF_{i,j}\subset\cY_i$  preserving the levels, cf.\ Lemmas \ref{lem:seq-V} and \ref{lem:triv-seq-Y}. The $N_i$ planes  $\cF_{i,(\alpha,\beta)}\subset V_i$ which carry the center $\mathfrak{S}_i=\cup_{l_h(\alpha)\ge i+1} L_{i,(\alpha,\beta)}$ of the blowup $\nu_{i+1}\colon V_{i+1}\to V_i$ are situated on the top level $i$. 
Hence their images $\tilde\cF_{i,j}$ are as well components of $z^*(0)$ in $\cY_i$ on the top level $i$. These  $N_i$ affine planes 
$\tilde\cF_{i,j}$ do not coincide, in general, with  the $N_i$ top level components of  $z^*(0)$ in $\cY_i$ which carry the center $\cS_i$ of the affine modification $s_{i+1}\colon\cY_{i+1}\to\cY_i$ from \eqref{eq: triv-seq-cY}. However, there is an automorphism $T$ of the bush $\Gamma_0(\pi_{Y_i})=(\tilde\Gamma)_{\le i}$  sending the $N_i$ branches of height $i$ carrying the vertices which correspond to the image of $\mathfrak{S}_i$ to the $N_i$ branches of height $i$ carrying the vertices which correspond to $\cS_i$. By Theorem I.5.7,  $T$ is induced by an automorphism $\tilde\tau\in\SAut_B(\cY_i)$. Applying $\tilde\tau$ one may suppose that the same collection $\mathfrak{F}_i$ of top level $i$ components of $z^*(0)$  in $\cY_i$ supports  the center $\cS_i$ of modification $s_{i+1}$ and the image of the center $\mathfrak{S}_i$ of $\nu_{i+1}$. Now each component $\tilde\cF_{i,j}\in\mathfrak{F}_i$ carries two distinct embedded affine lines  contained in  $\mathfrak{S}_i$ and $\cS_i$, respectively. By the relative Abhyankar-Moh-Suzuki Theorem (see Proposition I.4.14) 
there exists an automorphism $\eta_i\in\SAut_B(\cY_i)$ which sends the image of $\mathfrak{S}_{i}$ to $\cS_{i}$ and preserves the levels of the components of $z^*(0)$. Composing $\varphi_i$ and $\eta_i$ one may suppose that $\varphi_i(\mathfrak{S}_{i})=\cS_{i}$ while $\varphi_i$ is  still level-preserving. By Lemma I.1.5, $\varphi_i$ admits a lift to a level-preserving isomorphism $\varphi_{i+1}\colon V_{i+1}\stackrel{\cong_B}{\longrightarrow}\cY_{i+1}$ which fits in a commutative diagram
\be
 \bdi\la{diagr:blowups}
V_{i+1}&\rTo<{\varphi_{i+1}}>{\cong_B} &  \cY_{i+1} \\
\dTo<{\nu_{i+1}} &  & \dTo>{s_{i+1}} \\
V_i &\rTo>{\varphi_i}<{\cong_B} & \cY_i\, 
\edi
\ee
Passing from $V_i$ to $V_{i+1}$ increases by 1 the levels of the components of $z^*(0)$ which carry the center of the blowup while the levels of the remaining components do not change. The same is true for the passage from $\cY_i$ to $\cY_{i+1}$. It follows that $\varphi_{i+1}$ still preserves the levels. This concludes the  inductive step. 
\qed

\section{Cylinders over Danielewski-Fieseler surfaces}
In this section we restrict to the GDF surfaces with a unique special fiber. 
Our main result (Theorem \ref{thm: 8.main}) says that, up to isomorphism of cylinders over the base, 
it suffices to consider only the surfaces whose fiber tree of the unique special fiber is a bush.

\bdefi\la{def: one-fiber} We say that a marked GDF surface $\pi\colon X\to B$ 
over a smooth affine curve $B$ with a marking $z\in\cO_B(B)\setminus\{0\}$
is a \emph{marked Danielewski-Fieseler surface} if
\begin{itemize}\item $z^*(0)$ is a reduced divisor supported at a single point $b_0\in B$;
\item the restriction $\pi|_{X\setminus\pi^{-1}(b_0)}\colon X\setminus\pi^{-1}(b_0)\to B\setminus\{b_0\}$ 
is a trivial line bundle. \end{itemize} Under these assumptions there exists a trivializing sequence 
\eqref{eq: trivializing-seq-aff-modif} such that for any $i=1,\ldots,m-1$ the divisor of 
the fibered modification $\rho_{i+1}\colon X_{i+1}\to X_i$
is $z^{*}(0)$ on $X_i$.
\edefi

\brem\la{rem: one-fiber} The proofs of Propositions \ref{9.1} and \ref{9.2} go verbatim 
after replacing everywhere the Danielewski-Fieseler surfaces over the pair $(\A^1,0)$   by the marked Danielewski-Fieseler surfaces over the pair $(B,b_0)$ and replacing every affine space $\A^s$ by the product $B\times\A^{s-1}$. 
\erem

The main result in this section is the following.

\bthm\la{thm: 8.main} Let $\pi_X\colon X\to B$ be a marked Danielewski-Fieseler surface as in Definition  {\rm \ref{def: one-fiber}}. 
Then there is another marked Danielewski-Fieseler surface $\pi_Y\colon Y\to B$ with the same marking
such that \begin{itemize} \item 
the fiber tree $\Gamma_{b_0}(\pi_Y)$ is a bush with ${\rm tp}\,(\Gamma_{b_0}(\pi_Y))={\rm tp}\,(\Gamma_{b_0}(\pi_X))$;
\item there is an isomorphism over $B$ of cylinders  $\cX\stackrel{\cong_B}{\longrightarrow}\cY$ which sends the components of  $\pi_{\cX}^{-1}(b_0)$ to components of $\pi_{\cY}^{-1}(b_0)$ preserving the levels.
\end{itemize}
\ethm

\bproof Consider a trivializing sequence \eqref{eq: seq-aff-modif-final} for $X=X_m$ with $\Gamma_{b_0}(\pi_{X_i})=\Gamma_{\le i}$ where $\Gamma=\Gamma_{b_0}(\pi_X)$. There is a unique, up to isomorphism, bush ${\rm bh}(\Gamma)$ such that $\tp({\rm bh}(\Gamma))=\tp(\Gamma)$.

 If $m\le 1$ then $\Gamma$ is already a bush. Suppose further that $m\ge 2$. 
Assume by induction that the assertion holds for $X_{m-1}$, that is, there is  a marked Danielewski-Fieseler surface $\pi_{Y_{m-1}}\colon Y_{m-1}\to B$ sharing the same marking $z\in\cO_B(B)$ with $\pi_X\colon X\to B$ where $z^*(0)=b_0$ such that 
\begin{itemize}\item
$\Gamma_{b_0}(\pi_{Y_{m-1}})={\rm bh}(\Gamma_{\le m-1})$ \footnote{Notice that ${\rm bh}(\Gamma_{\le i})\neq ({\rm bh}(\Gamma))_{\le i}$ in general.} and \item there is an isomorphism  $\varphi_{m-1}\colon\cX_{m-1}\stackrel{\cong_B}{\longrightarrow}\cY_{m-1}$ preserving the levels of special fiber components. 
\end{itemize} 

Recall that $\rho_m\colon X_m \to X_{m-1}$ is a  fibered modification along the reduced divisor $z^*(0)=\sum_{i=1}^N F_i$ on $X_{m-1}$ with a reduced center $S=\sum_{i=1}^N S_i$ where for every $i=1,\ldots,N$ either $S_i= F_i$ or $S_i\subset F_i$ is nonempty and finite. Let $\mathfrak{F}$ be the collection
  of the fiber components $F_i$ in $X_{m-1}$ on the top level $m-1$ with finite $S_i$.
Consider the affine modification  $$\tilde\rho_m\colon\cX_m \to \cX_{m-1}, \quad \tilde\rho_m=\rho_m\times\id_{\A^1}\,,$$ along the reduced divisor  $z^{*}(0)=\sum_{i=1}^N \cF_i$ on $\cX_{m-1}$ where $\cF_i=F_i \times \A^1\cong\A^2$
with the reduced center $\cS =\sum_{i=1}^N\cS_i$ where $\cS_i = S_i \times \A^1$, $i=1,\ldots,N$. 

Let $z^{*}(0)=\sum_{i=1}^N \cF'_i$ on $\cY_{m-1}$ where $ \cF'_i=\varphi_{m-1}(\cF_i)$, and let $\cS_i'=\varphi_{m-1}(\cS_i)\subset\cF_i'$. 
Applying a suitable automorphism $\tau\in\SAut_B \cY_{m-1}$ one may suppose that 
\be\la{eq:AMS} \cS_i'=S_i' \times \A^1 \subset F_i'\times \{0\}\subset\cF_i'\quad\forall F_i\in \mathfrak{F}\,.\ee
 Indeed, choose for any $F_i\in \mathfrak{F}$ a point $s_i\in S_i$. Due to the relative Abhyankar-Moh-Suzuki Theorem (see Proposition I.4.14) 
one can rectify  the images $\varphi_{m-1}(\{s_i\}\times\A^1)\subset\cF'_i\cong\A^2$ simultaneously for all $F_i\in \mathfrak{F}$. Then $\tau\circ\varphi_{m-1}$ sends any parallel line 
$\{t_i\}\times\A^1\subset\cF_i=F_i\times\A^1$ where $t_i\in S_i\setminus\{s_i\}$ to an affine line, say, $\{t_i'\}\times\A^1$ parallel to (and disjoint with) the image $\{s_i'\}\times\A^1=\tau\circ\varphi_{m-1}(\{s_i\}\times\A^1)$. 

Assuming \eqref{eq:AMS} perform the fibered modification $\rho_m'\colon Y_{m}\to Y_{m-1}$ along the reduced divisor $z^*(0)$ on $Y_{m-1}$ with the reduced center  $\sum_{i=1}^N S_i'$. It results in a marked Danielewski-Fieseler surface $\pi_{Y_{m}}\colon Y_m\to B$. 
Consider the induced affine modification of cylinders
$$\tilde\rho_m'\colon\cY_m \to \cY_{m-1}, \quad \tilde\rho_m'=\rho'_m\times\id_{\A^1}\,,
$$ 
along the reduced divisor  $z^{*}(0)=\sum_{i=1}^N \cF'_i$ on $\cY_{m-1}$ with the reduced center $\cS' =\varphi_{m-1}(\cS_i)=\sum_{i=1}^N\cS_i'$ satisfying  \eqref{eq:AMS}.
Since $\varphi_{m-1}$ sends the divisor and the center of the affine modification $\cX_m \to \cX_{m-1}$ to the ones of $\cY_m \to \cY_{m-1}$ then by Lemma I.1.5,
$\varphi_{m-1}$ admits a lift to an isomorphism
$\varphi_{m}\colon\cX_m\stackrel{\cong_B}{\longrightarrow}\cY_m$ which automatically preserves the levels.

The fiber tree $\Gamma_{b_0} (\pi_{Y_m})$ is a spring bush  with $\tp(\tilde\Gamma)=\tp(\Gamma)$. Due to Proposition \ref{9.2} and Remark \ref{rem: one-fiber} there exists a marked 
Danielewski-Fieseler surface $\pi_Y\colon Y\to B$ sharing the same marking  $z$ such that 
\begin{itemize}\item $\Gamma_{b_0}(\pi_Y)={\rm bh}(\Gamma_{b_0} (\pi_{Y_m}))$ is a bush and
\item there is an isomorphism of cylinders $\cY\stackrel{\cong_B}{\longrightarrow}\cY_m$.
\end{itemize}
Since $\cX=\cX_m\cong_B\cY_m$ under $\phi_m$ it follows that
\begin{itemize}\item $\tp(\Gamma_{b_0}(\pi_Y))=\tp(\Gamma)$ and
\item there is an isomorphism of cylinders $\cX\stackrel{\cong_B}{\longrightarrow}\cY$,
\end{itemize} as required. 
\eproof

\section{Proof of the main theorem}\la{sec:proof}

\subsection{Invariance of the type up to a linear equivalence} \la{sec:proof-only-if}

The `only if' part of Theorem \ref{thm: main} follows from the next  proposition.

\bprop\label{prop: shifted-type} Consider two GDF surfaces $\pi_X\colon X\to B$ 
and $\pi_{Y}\colon Y\to B$ over the same base $B$. Assume that there is an isomorphism  over $B$ of cylinders
$\varphi\colon\cX\stackrel{\cong_B}{\longrightarrow}\cY$. Let $\tau\colon {\rm DF}(\pi_X)\stackrel{\cong_B}{\longrightarrow}{\rm DF}(\pi_Y)$ be the
induced isomorphism of the Danielewski-Fieseler quotients, 
 see Lemma {\rm \ref{lem:elem}}. 
Then the type divisors $\tp.\div(\pi_X)$ and $\tau^*(\tp.\div(\pi_Y))$ on ${\rm DF}(\pi_X)$ 
are linearly equivalent.
\eprop

\bproof Consider the birational morphisms 
$$\sigma_X\colon X\to X_0=B\times\A^1\quad\mbox{and}\quad \sigma_Y\colon Y\to Y_0=B\times\A^1\,$$ 
which are compositions of the contractions $\rho_i$  in \eqref{eq: trivializing-seq-aff-modif}. Consider also 
the induced birational morphisms of cylinders
\be\la{eq:contractions} \sigma_{\cX}\colon \cX\to B\times\A^2,\,\,\,\sigma_{\cX}=\sigma_X\times\id_{\A^1},\quad\mbox{and}\quad \sigma_{\cY}\colon \cY\to B\times\A^2,\,\,\,\sigma_{\cY}=\sigma_Y\times\id_{\A^1}\,.\ee
There is a birational map 
$\psi\colon  B\times\A^2\dashrightarrow B\times\A^2$ fitting 
in the commutative diagram 
\be\label{diagr: star}
 \bdi
 \cX&\rTo<{\phi}>{\cong_B} &  \cY\\
\dTo<{\sigma_{\cX}} &  & \dTo>{\sigma_{\cY}} \\
B\times\A^2 &\stackrel{\psi}{---->} &  B\times\A^2\,. 
\edi
\ee 
Fix a common marking $z\in\cO_B(B)\setminus\{0\}$ for $\pi_X\colon X\to B$ and $\pi_Y\colon Y\to B$ with $z^*(0)=b_1+\ldots+ b_n$. 
Letting $B^*=B\setminus\{b_1,\ldots,b_n\}\subset B$, $\psi$ restricts to an automorphism  over $B^*$ of $B^*\times\A^2$. Letting $\A^2=\Spec \mathbb{k}[u,v]$ the Jacobian $J(\psi)(b)$   of $\psi|_{\{b\}\times\A^2}\in\Aut\A^2$  is a rational function on $B$ without zeros and poles in $B^*$. 

\smallskip

\noindent {\bf Claim.} \emph{Consider a fiber component $\cF=F\times\A^1\subset \pi_{\cX}^{-1}(b_i)$ and its image $\cF'=\varphi(\cF)\subset\pi_{\cY}^{-1}(b_i)$. Let $l$ and $l'$ be the levels of $F$ and $F'$, respectively. Then one has 
\be\la{eq:order} {\ord}_{b_i} J(\psi)=l'-l\,.\ee This order does not depend on the choice of a component $\cF$ of $\pi_{\cX}^{-1}(b_i)$.}

\smallskip

\noindent \emph{Proof of the claim}.
Let  $\cP$ ($\cP'$, respectively)  be the path of length $l$ ($l'$, respectively) in $\Gamma_{b_i}(\pi_X)$ ($\Gamma_{b_i}(\pi_{Y})$, respectively) joining the leaf $\bar F$ ($\bar F'$, respectively) with the root. The chains $\cP$ and $\cP'$ are the dual graphs of the total transforms of $\{b_i\}\times\PP^1\subset\bar B\times\PP^1$ in \eqref{eq: trivializing-seq-blowups} under certain sequences of blowups with centers in some points $\alpha,\alpha'\in\{b_i\}\times\A^1$, respectively, and infinitely near points. 

Choose local coordinates 
$(z,u)$ near the fiber $\{b_i\}\times\A^1$ in $B_i\times\A^1$ so that $\alpha=(0,0)\in\{b_i\}\times\A^1$. The chain of affine modifications which corresponds to $\cP$ yields in suitable local coordinates the chain of extensions 
\be\la{eq:chain-of-extensions} \mathbb{k}[z,u]\subset \mathbb{k}[z,u/z]\subset \mathbb{k}[z,u/z^2]\subset\ldots\subset  \mathbb{k}[z,u/z^l]\,.\ee
Letting $B_i=B^*\cup\{b_i\}$ consider the standard neighborhoods $U_F\cong_{B_i} B_i\times\A^1$ of $F$ in $X$ and $U_{F'}\cong_{B_i} B_i\times\A^1$ of $F'$ in $Y$, see Proposition I.3.3, along with their cylinders
$$\cU_\cF=U_F\times \A^1\subset\cX\quad\mbox{and}\quad \cU_{\cF'}=U_{F'}\times \A^1\subset\cX'\quad\mbox{where}\quad \cU_\cF\cong_{B_i}\cU_{\cF'}\cong_{B_i} B_i\times\A^2\,.$$ 
Let $\Omega_i=B_i\times\A^2$ be the standard neighborhood of $\{b_i\}\times\A^2$ in $B_i\times\A^2$.
Due to \eqref{eq:chain-of-extensions} the restriction $\sigma_{\cX}|_{\cU_\cF}\colon\cU_\cF\to\Omega_i$ can be given in suitable local coordinates $(z,t,v)$ in $\cU_\cF$ near $\cF$ and $(z,u,v)$ in $\Omega_i$ near $\{b_i\}\times\A^2$ as
$$\sigma_{\cX}\colon (z,t,v)\mapsto (z,z^lt,v)\,,$$ cf.\ Corollary \ref{cor:loc-coord}(a).
Similarly, the restriction $\sigma_{\cY}|_{\cU_{\cF'}}\colon\cU_{\cF'}\to\Omega_{i}'$ can be given in suitable local coordinates as
$$\sigma_{\cY}\colon (z,t',v')\mapsto (z,z^{l'}t',v')\,.$$
 In these local coordinates one obtains
$$\det(d_{(t,v)}\sigma_{\cX})=z^{l}\quad\mbox{and}\quad\det(d_{(t',v')}\sigma_{\cY})=z^{l'}\,.$$ Since $\det(d_{(t,v)}\varphi)$ is an invertible function on $B$ one has $J(\psi)(z)\sim z^{l'-l}$ near $b_i$. This yields \eqref{eq:order}. Now the claim follows. 

\smallskip

Using \eqref{eq:order} one concludes that 
\be\la{eq: jac} \tau^*(\tp.\div(\pi_Y))-\tp.\div(\pi_X)=p^*(\div J(\psi))\,\ee where $p\colon {\rm DF}(\pi_X)\to B$ is the natural projection. Now the proof is completed.
 \eproof

The following corollary is immediate from \eqref{eq: jac}.

\bcor\la{cor:jac} Under the assumptions of Proposition {\rm\ref{prop: shifted-type}}, $\psi\in\Aut_B(B\times\A^2)$  if and only if 
$\tp.\div(\pi_X)=\tau^*(\tp.\div(\pi_Y))$.
\ecor

\subsection{Special isomorphisms and regularization}\label{ss:regularization}
In this subsection we finish the proof of Theorem \ref{thm: main}. 
We need the following notions. 

\bdefi[\emph{Special isomorphisms}]\la{def: special-isomorphism}
Given marked GDF surfaces $\pi_X\colon X\to B$ and $\pi_{Y}\colon Y\to B$ sharing a common marking $z\in\cO_B(B)\setminus\{0\}$ consider trivializing sequences \eqref{eq: triv-seq-cX-h} of cylinders along with the corresponding birational morphisms $\sigma_{\cX}$ and $\sigma_{\cY}$ as in \eqref{eq:contractions}. Given an isomorphism $\phi\colon\cX\stackrel{\cong_B}{\longrightarrow}\cY$ consider a  birational $B$-automorphism  $\psi=\sigma_{\cY}\circ\varphi\circ\sigma_{\cX}^{-1}$ of $B\times\A^2$ biregular off $z^{-1}(0)$ and fitting in diagram 
\eqref{diagr: star}.
 We say that $\phi$ is \emph{special} if $\psi\in\SAut_{B^*}(B^*\times\A^2)$ where $B^*=B\setminus z^{-1}(0)$ and \emph{sub-special} if $\psi\in\SAut_{B'}(B'\times\A^2)$ for some Zariski open dense subset $B'\subset B$.
\edefi

One has the following criterion for an isomorphism to be sub-special. 

\blem\la{lem: specialty} An isomorphism $\phi\colon\cX\stackrel{\cong_B}{\longrightarrow}\cY$ is sub-special if and only if
$\psi$ in {\rm (\ref{diagr: star})} verifies \be\la{eq: Jacobi} J(\psi)(b)={\rm Jac}(\psi|_{\{b\}\times\A^2})= 1\quad\forall b\in B^*\,.\ee If \eqref{eq: Jacobi} holds then one has a factorization
\be\la{eq: factoriz} \psi=\prod_{i=1}^{N}\exp(\p_i)\,\ee where $\p_1,\ldots,\p_N$ are locally nilpotent derivations of $\cO_{B'}(B')[u,v]$ for an open dense subset $B'\subset B$. 
\elem

\bproof
The `only if' part of the first assertion is well known, see, e.g., \cite[Lem.\ 4.10]{AFKKZ}. Assume further that $\psi$ satisfies (\ref{eq: Jacobi}). Let $\cB=\cO_{B^*}(B^*)$, and let $K=\Frac\cO_B(B)$ be the function field of $B$. Since $\psi\in\Aut_{\cB}\cB[u,v]$ and 
$J(\psi)=1$ one has  $\psi\in\SAut_{K}K[u,v]$, 
see \cite[Prop.\ 9]{FL} or \cite[Example 2.1]{Po}. Note that the base field $K$ in \cite{FL} and \cite{Po}  is supposed to be algebraically closed of characteristic zero. However, due to the van der Kulk version of the Jung Theorem (see, e.g., \cite{Wr}) the algebraic closeness assumption is superfluous.

 The group $\SAut_{K}K[u,v]$
is generated by the replicas of the locally nilpotent derivations $\p/\p u$ and $\p/\p v$ (\cite[Example 2.1]{Po}). Therefore, one has a decomposition 
\be\la{eq: psi-decomp} \psi=\prod_{i=1}^{m}\exp(f_i\p/\p u)\exp(g_i\p/\p v)=\prod_{j=1}^{N}\exp(\p_j),\quad N=2m\,,\ee where $f_i\in K[v]$, $g_i\in K[u]$, and the $\p_j$ are locally nilpotent $K$-derivations of $K[u,v]$, $i=1,\ldots,n,\,j=1,\ldots,N$. This yields (\ref{eq: factoriz}). 
\eproof

A regularization procedure described below allows to replace certain isomorphism of cylinders by special ones. 

\bprop\la{prop: jac} Under the assumptions of Proposition {\rm\ref{prop: shifted-type}} suppose in addition that $\tp.\div(\pi_X)=\tau^*(\tp.\div(\pi_{Y}))$. Then there exists 
a  special isomorphism $\tilde\phi\colon\cX\stackrel{\cong_B}{\longrightarrow}\cY$.
\eprop

\bproof 
Due to Corollary \ref{cor:jac}, $\psi$  in \eqref{diagr: star} is biregular, and so, the Jacobian $J(\psi)(b)$ is a non-vanishing regular function on $B$. Consider the automorphism $$\eta\in\Aut_B\cY,\qquad \eta\colon (y,v)\mapsto (y,v/J(\psi))\quad\mbox{for}\quad (y,v)\in \cY=Y\times\A^1\,.$$ Then $\hat\phi:=\eta\circ\phi\colon \cX\stackrel{\cong_B}{\longrightarrow}\cY$ and $\hat\psi:=\sigma_{\cY}\circ\hat\varphi\circ\sigma_{\cX}^{-1}\in\Aut (B\times\A^2)$ fit in \eqref{diagr: star}. One has $J(\hat \psi)\equiv 1$. By Lemma \ref{lem: specialty}, $\hat\phi$ is sub-special. To simplify the notation we will
suppose that $\phi$ is. 

According to Lemma \ref{lem: specialty}, $\psi$ can be factorized as in (\ref{eq: psi-decomp}).
For a natural $s\gg 1$
choose a function $h\in\cO_B(B)\setminus\{0\}$ such that 
\begin{itemize}
\item $hf_i\in\cO_{B^*}(B^*)[u]$ and $hg_i\in\cO_{B^*}(B^*)[v]$ where $f_i,g_i$ are as in \eqref{eq: psi-decomp}, $i=1,\ldots,m$;
\item $h-1$ vanishes to order $s+M$ at $b_1,\ldots,b_n$ where $M$ is the maximal order of pole at $b_1,\ldots,b_n$ of the coefficients of $f_i,g_i$, $i=1,\ldots,m$. 
\end{itemize}
Replacing the factors of $\psi$ in  (\ref{eq: psi-decomp}) by their $h$-replicas yields an automorphism  
\be\la{eq: fact-psi-h} \psi_h=\prod_{i=1}^{N}\exp(h\p_i)\in {\SAut}_{B^*}(B^*\times\A^2),\qquad N=2m\,.\ee

\smallskip

\noindent {\bf Claim.} \emph{Letting $\alpha=\psi_h\circ\psi^{-1}$ one has  $\alpha\equiv\id\mod z^{s}$ and $\alpha^{-1}\equiv\id\mod z^{s}$.}

\smallskip

\noindent\emph{Proof of the claim.} We prove the first congruence, the proof of the second one being similar. Let
$$\alpha_j=\left(\prod_{i=1}^{j}\exp(h\p_i)\right)\left(\prod_{i=1}^{j}\exp(\p_i)\right)^{-1}=\left(\prod_{i=1}^{j-1}\exp(h\p_i)\right)\exp((h-1)\p_j)\left(\prod_{i=1}^{j-1}\exp(\p_i)\right)^{-1}\,.$$
Since $(h-1)\p_j\equiv 0\mod z^s$ one has $$\exp((h-1)\p_j)-\id\equiv 0\mod z^s\,.$$ Hence $$\alpha_1\equiv\id\mod z^{s}\quad\mbox{and}\quad \alpha_j-\alpha_{j-1}\equiv 0\mod z^s\,.$$
By recursion one gets $\alpha=\alpha_N\equiv\id\mod z^{s}$, as needed.

\smallskip

Notice that $\alpha$ satisfies the conditions of Lemma I.1.6
with respect to the affine modification
$\sigma_{\cY}\colon \cY\to B\times\A^2$ along the divisor $(z^t)^*(0)$ where $t={\rm ht}(\cD(\pi_X))$. Indeed, $\alpha_*(z)=z$, and $\alpha\equiv\id\equiv\alpha^{-1}\mod z^{t\lfloor s/t\rfloor}$ where $s\gg t$. 
By Lemma I.1.6,
$\alpha\in \SAut_B(B\times\A^2)$ admits a lift $\tilde\alpha\in\Aut_B\cY$ such that $\tilde\alpha\equiv\id \mod z^{t\lfloor s/t\rfloor-t}$. Letting $\tilde\phi=\tilde\alpha\circ\phi\colon\cX\to\cY$ and $\tilde\psi=\alpha\circ\psi=\psi_h\in \SAut_{B^*}(B^*\times\A^2)$  yields a pair $(\tilde\phi, \tilde\psi)$ fitting in  \eqref{diagr: star}. So, the isomorphism $\tilde\phi\colon\cX\stackrel{\cong_B}{\longrightarrow}\cY$ is  special.
\eproof

\noindent \emph{Proof of the `if' part of Theorem {\rm \ref{thm: main}.}} It is worth to start by recalling the setup. We consider two marked GDF surfaces $\pi_X\colon X\to B$ and $\pi_Y\colon Y\to B$  over the same base $B$ and with the same marking $z\in\cO_B(B)\setminus\{0\}$ where $z^*(0)=b_1+\ldots+b_n$. We assume that there is an isomorphism
$\tau\colon {\rm DF}(\pi_X)\stackrel{\cong_B}{\longrightarrow}{\rm DF}(\pi_Y)$ such that $\tp.\div(\pi_X)\sim \tau^*(\tp.\div(\pi_Y))$ on ${\rm DF}(\pi_X)$. 
We must show that under these assumptions there exists an isomorphism $\varphi\colon\cX\stackrel{\cong_B}{\longrightarrow}\cY$.

By Lemma \ref{lem: reduction} one may suppose that 
\be\la{eq:tp-equal} \tp.\div(\pi_X)=\tau^*(\tp.\div(\pi_Y))\,.\ee 
We proceed by induction on the number $n=\card z^{-1}(0)$ of special fibers. 
If $n=1$ then $\pi_X\colon X\to B$ and $\pi_Y\colon Y\to B$ are marked Danielewski-Fieseler surfaces as in Definition \ref{def: one-fiber}. By Theorem \ref{thm: 8.main} there exist two other marked Danielewski-Fieseler surfaces $\pi_{X'}\colon X'\to B$ and $\pi_{Y'}\colon Y'\to B$ with the same marking such that
\begin{itemize}
\item there are isomorphisms  $\cX'\cong_B\cX$ and $\cY'\cong_B\cY$, and 
\item the fiber trees
$\Gamma_{b_1}(\pi_{X'})$ and $\Gamma_{b_1}(\pi_{Y'})$ are bushes with $$\tp(\Gamma_{b_1}(\pi_{X'}))=\tp(\Gamma_{b_1}(\pi_{X}))\quad\mbox{and}\quad \tp(\Gamma_{b_1}(\pi_{Y'}))=\tp(\Gamma_{b_1}(\pi_{Y}))\,.$$
\end{itemize} 
Then \eqref{eq:tp-equal} implies 
$$\tp(\Gamma_{b_1}(\pi_{X'}))=\tp(\Gamma_{b_1}(\pi_{Y'}))\,.$$ 
Therefore, the bushes $\Gamma_{b_1}(\pi_{X'})$ and $\Gamma_{b_1}(\pi_{Y'})$ are isomorphic. By Theorem I.5.7 
there is an isomorphism of cylinders $\cX'\cong_B\cY'$, hence also an isomorphism $\varphi\colon\cX\stackrel{\cong_B}{\longrightarrow}\cY$. This proves the assertion for $n=1$.

Suppose the assertion holds if $\card z^{-1}(0)\le n-1$. Consider further
 the case  $\card z^{-1}(0)=n$. Let $$B_0=B\setminus \{b_2,\ldots,b_n\}\quad\mbox{and}\quad B_1=B\setminus\{b_1\}\quad\mbox{so that}\quad  B_0\cup B_1=B\quad\mbox{and}\quad B_0\cap B_1=B^*\,.$$ Let $\pi_{X_0}\colon X_0\to B_0$ and $\pi_{X_1}\colon X_1\to B_1$ be the restrictions of $\pi_X$ and $\pi_Y$ over $B_0$ and $B_1$, respectively. Define in a similar way $\pi_{Y_0}\colon Y_0\to B_0$ and $\pi_{Y_1}\colon Y_1\to B_1$. By the inductive hypothesis there are commutative diagrams (cf.\ \eqref{diagr: star})
 \be\label{diagr: star-two}
 \bdi
 \cX_0&\rTo<{\phi_0}>{\cong_{B_0}} &  \cY_0\\
\dTo<{\sigma_{\cX}|_{\cX_0}} &  & \dTo>{\sigma_{\cY}|_{\cY_0}} \\
B_0\times\A^2 &\stackrel{\psi_0}{---->} &  B_0\times\A^2\, 
\edi
\qquad\mbox{and}\qquad
 \bdi
 \cX_1&\rTo<{\phi_1}>{\cong_{B_1}}  &  \cY_1\\
\dTo<{\sigma_{\cX}|_{\cX_1}} &  & \dTo>{\sigma_{\cY}|_{\cY_1}} \\
B_1\times\A^2 &\stackrel{\psi_1}{---->} &  B_1\times\A^2\, 
\edi
\ee 
By virtue of Proposition \ref{prop: jac} one may suppose that  both $\varphi_0$ and $\varphi_1$ in \eqref{diagr: star-two} are special, see Definition \ref{def: special-isomorphism}.
To cook up an isomorphism $\varphi$ over $B$ using $\varphi_0$ and $\varphi_1$ we apply the same kind of regularization as in the proof of Proposition \ref{prop: jac}. 

\medskip

\noindent {\bf Claim 1}. 
\emph{Given $s\gg 1$ there exist $\psi'_0,\psi'_1\in\SAut_{B}(B\times\A^2)$ satisfying 
\begin{itemize}
\item[{\rm (i)}] $\psi_0'\equiv\psi_0\mod z^s$ near $\{b_1\}\times\A^2$ and $\psi'_0\equiv\id\mod z^s$ near $\{b_i\}\times\A^2$, $i=2,\ldots,n$; 
\item[{\rm (ii)}] $\psi'_1\equiv\id\mod z^s$ near $\{b_1\}\times\A^2$ and $\psi'_1\equiv\psi_1\mod z^s$ near $\{b_i\}\times\A^2$, $i=2,\ldots,n$.
\end{itemize}}

\medskip

\noindent \emph{Proof of Claim $1$}. It suffices to prove (i), the proof of (ii) being similar. Likewise in the proof of Proposition \ref{prop: jac} we replace the factors $\exp(\p_j)$ in the factorization 
$\psi_0=\prod_{j}\exp(\p_j)\in{\SAut}_{B^*}(B^*\times\A^2)$ as in (\ref{eq: fact-psi-h}) by their replicas $\exp(h\p_j)$ where $h\in\cO_B(B)$ satisfies
\begin{itemize}\item $h-1$ vanishes to  order $s'\gg s$ at $b_1$;
\item $h$ vanishes to order $s$ at $b_2,\ldots,b_n$. 
\end{itemize}
This gives an automorphism $\psi_0'=\prod_{j}\exp(h\p_j)\in\SAut_{B}(B\times\A^2)$ 
verifying (i), as desired.  

\medskip

The following claim ends the proof of the theorem. 

\medskip

\noindent {\bf Claim 2}. 
\emph{The composition $\psi=\psi_0'\circ\psi_1'\in\SAut_{B}(B\times\A^2)$ admits a lift to an isomorphism $\varphi\colon\cX\stackrel{\cong_B}{\longrightarrow}\cY$. 
}

\medskip

\noindent \emph{Proof of Claim $2$}. The affine modifications $\sigma_\cX\colon \cX\to B\times\A^2$ and $\sigma_\cY\colon \cY\to B\times\A^2$ along the divisors $z^*(0)$ on $\cX$ and $\cY$, respectively, have for their centers certain subschemas $\cS_{\cX},\,\cS_{\cY}$ of the $r$th infinitesimal neighborhood of the divisor $z^{-1}(0)$ on $B\times\A^2$ where $r\ge 1$ is such that $z^r$ belongs to the corresponding defining ideals $I(\cS_\cX), I(\cS_\cY)\subset\cO_{B\times\A^2}(B\times\A^2)$. By Lemma I.1.5, $\psi$ can be lifted to an isomorphism $\varphi\colon\cX\stackrel{\cong_B}{\longrightarrow}\cY$ provided $\psi^*$ sends $I(\cS_\cY)$ to $I(\cS_\cX)$, or, which is equivalent, if $\psi(\cS_{\cX})=\cS_{\cY}$. 

For  $i=0,1$ the centers $\cS_{\cX_i}$ and $\cS_{\cY_i}$ of the affine modifications $\sigma_{\cX}|_{\cX_i}$ and $\sigma_{\cY}|_{\cY_i}$ coincide with the restrictions $\cS_{\cX}|_{\cX_i}$ and  $\cS_{\cY}|_{\cY_i}$, respectively. Since 
$\psi_i$ admits a lift $\varphi_i$ (see \eqref{diagr: star-two}) one has $\psi_i(\cS_{\cX_i})=\cS_{\cY_i}$, $i=0,1$. 

Due to (i) and (ii) for $s\ge r$ one has $\psi|_{B_i\times\A^2}\equiv \psi_i|_{B_i\times\A^2}\mod z^r$ for $i=0,1$, that is, these automorphisms coincide in the $r$th infinitesimal neighborhood of $z^{-1}(0)$ in $B_i\times\A^2$. 
It follows that $\psi(\cS_{\cX_i})=\cS_{\cY_i}$, $i=0,1$. Finally one has $\psi(\cS_{\cX})=\cS_{\cY}$, as required. \qed

\medskip

\section{On moduli spaces of GDF surfaces} \la{sec:deformations}

We conclude the paper by constructions of a versal deformation family and an affine coarse moduli space of marked GDF surfaces with a given marking, a given cylinder, and a given graph divisor.

\subsection{Coarse moduli spaces of GDF surfaces}\la{ss:coarse}
\begin{defi}\label{def:deform-family} Consider a marked GDF surface $\pi_X\colon X\to B$ with a marking $z\in\cO_B(B)\setminus\{0\}$, a trivializing sequence \eqref{eq: trivializing-seq-aff-modif}, and a graph divisor $\cD=\cD(\pi_X)$. 
By a \emph{family of GDF surfaces of type $(B,z,\cD)$} we mean a pair of smooth morphisms of quasiprojective schemes $\mathfrak{X}\to \mathcal{S}$ and $\pi_{\mathfrak{X}}\colon \mathfrak{X}\to B$ such that 
\begin{itemize}
\item  for each  point $\mathfrak{s}\in\mathcal{S}$ the fiber $X(\mathfrak{s})$ of $\mathfrak{X}\to \mathcal{S}$  over $\mathfrak{s}$ is reduced; 
\item  the specialization $\pi(\mathfrak{s})\colon X(\mathfrak{s})\to B$ over $\mathfrak{s}$ is a marked GDF surface with the marking $z$ and the graph divisor $\cD$;
\item there is a point $\mathfrak{s}_0\in \mathcal{S}$ such that the specialization over $\mathfrak{s}_0$ yields the initial marked GDF surface $\pi_{X}\colon X\to B$.   
\end{itemize}

We say that $\mathfrak{X}\to \mathcal{S}$ is an (\emph{affine}) \emph{deformation family of marked GDF surfaces} if both $\mathfrak{X}$ and $\mathcal{S}$ are smooth (affine) varieties and the morphism $\mathfrak{X}\to\mathcal{S}$ is a submersion which extends to a proper deformation family of SNC completions of GDF surfaces over $\bar B$ such that the corresponding family of extended divisors $(D_{\rm ext}(\mathfrak{s}))_{\mathfrak{s}\in\mathcal{S}}$  over $\mathcal{S}$ is locally trivial. This yields a locally trivial family of graph divisors $(\cD(\mathfrak{s})=\cD(\pi_{X(\mathfrak{s})}))_{\mathfrak{s}\in\mathcal{S}}$. The monodromy group of the latter family is a subgroup of the finite group $$\Aut_B(\cD(\mathfrak{s}_0))=\prod_{i=1}^n \Aut(\Gamma_i(\mathfrak{s}_0))\,.$$ We say that the family of graph divisors $(\cD(\mathfrak{s}))_{\mathfrak{s}\in\mathcal{S}}$ is \emph{trivial} if its monodromy group is.
\end{defi}

The Grothendieck theory of moduli spaces (\cite{Gro}) and its versions deal with proper schemes, or pairs of proper schemes. In our particular non-proper setting we adopt the following simplified definition.

\begin{defi}\la{def:coarse-moduli} Consider a triplet $(B,z,\cD)$ where $B$ is a smooth affine curve, $z\in\cO_B(B)\setminus\{0\}$ is a marking, and $\cD=\sum_{i=1}^n \Gamma_ib_i$ is a graph divisor with trees $\Gamma_1,\ldots,\Gamma_n$ supported by $z^{-1}(0)=\{b_1,\ldots,b_n\}\subset B$. A scheme $\mathcal{C}(B,z, \cD)$ will be called a \emph{coarse moduli space of GDF surfaces of type $(B,z,\cD)$} if the following hold.
\begin{itemize}\item To any marked GDF surface $\pi_X\colon X\to B$ of type $(B,z,\cD)$ there corresponds a unique point $\mathfrak{c}(\pi_X)\in\mathcal{C}(B,z, \cD)$, and vice versa, any point $\mathfrak{c}\in\mathcal{C}(B,z, \cD)$ corresponds to a unique, up to an isomorphism over $B$, marked GDF surface $\pi_{X(\mathfrak{c})}\colon X(\mathfrak{c})\to B$ of type $(B,z,\cD)$;
\item for any deformation family of GDF surfaces
$\mathfrak{X}\to\mathcal{S}$ of type $(B,z,\cD)$ the correspondence $\mathcal{S}\to \mathcal{C}(B,z, \cD)$ sending a point $\mathfrak{s}\in\mathcal{S}$ to $\mathfrak{c}(\pi_{X(\mathfrak{s})})\in\mathcal{C}(B,z, \cD)$ is a morphism. 
\end{itemize}
\end{defi}

Under certain restrictions there exists a coarse moduli space of GDF surfaces.

\bthm\la{thm:moduli} Consider a triplet $(B,z,\cD)$ as in Definition {\rm \ref{def:coarse-moduli}}.
Suppose that $\cO_B(B)^\times = \mathbb{k}^*$, that is, $B$ does not admit any non-constant invertible regular function. Then the following hold.
\begin{itemize}
\item[{\rm (a)}] 
There exists a coarse moduli space $\mathcal{C}(B,z, \cD)$ of marked GDF surfaces of type $(B,z,\cD)$. This space $\mathcal{C}(B,z, \cD)$ is an affine  variety with at most quotient singularities.
\item[{\rm (b)}] 
The cylinders over any two surfaces in $\mathcal{C}(B,z, \cD)$ are isomorphic over $B$. 
\item[{\rm (c)}] If $\pi\colon X\to B$ is a marked GDF surface of type $(B,z,\cD)$ and $X$ is not a Zariski $1$-factor, that is, $ \cD(\pi_X)\cong\cD$ is not a chain divisor then there exists a sequence of marked GDF surfaces $\pi_{X^{(k)}}\colon X^{(k)}\to B$ with the given marking $z$ and the cylinders isomorphic over $B$ to $X\times\A^1$ 
such that
$$
\dim \,\mathcal{C}(B,z, \cD(\pi_{X^{(k)}}))\to\infty\quad\mbox{when}\quad k\to\infty\,.$$ 
\end{itemize}
\ethm

The proof of Theorem \ref{thm:moduli} is  done at the end of the section. 

\subsection{The automorphism group of a GDF surface}\la{ss:Aut}

\begin{nota}\label{nota:Ru} Given an $\A^1$-fibered surface $\pi_X\colon X\to B$  we let $\mathcal{U}(\pi_X)=R_u(\Aut_B(X))$ be the (normal) subgroup of $\Aut_B(X)$ generated by all the $\mathbb{G}_a$-actions on $X$ along the fibers of $\pi_X$. For a GDF surface $\pi_X\colon X\to B$ with a trivializing sequence \eqref{eq: trivializing-seq-aff-modif} we let $\mathcal{U}_l=\mathcal{U}(\pi_{X_l})$, $l=1,\ldots,m$ where $\mathcal{U}_m=\mathcal{U}(\pi_X)$.
\end{nota}

The next proposition can be deduced from   Theorems 6.3 and 8.24 in \cite{KPZ} and their corollaries. However, in our particular case we prefer to give a simple direct argument.

\begin{prop}\label{prop:aut-GDF} Assume that $\cO_B(B)^\times=\mathbb{k}^*$. Given a  GDF surface  $\pi_X\colon X\to B$ with a trivializing sequence \eqref{eq: trivializing-seq-aff-modif} the following hold.
\begin{itemize}
\item[{\rm (a)}]
There are natural inclusions \be\la{eq:nat-incl}
\Aut_B(X)=\Aut_B(X_m)\subset\Aut_B(X_{m-1})\subset\ldots\subset\Aut_B(X_0)=\Aut_B(B\times\A^1)\,.\ee

\item[{\rm (b)}] One has
$\Aut_B(B\times\A^1)=\mathbb{U}_0\rtimes\mathbb{G}_m$ where $\mathbb{U}_0\cong\cO_B(B)$ viewed as a vector group. 

\item[{\rm (c)}] Let $m_{i,l}=\h((\Gamma_i)_{\le l})$, $i=1,\ldots,n$. For any $l=0,\ldots,m$ one has $$\mathbb{U}_l=\mathbb{U}_0\cap \Aut_B(X_l)\cong H^0(B,-D_l)\quad\mbox{where}\quad D_l=\sum_{i=1}^n m_{i,l}b_i\,.$$

\item[{\rm (d)}] 
Let $C_{l-1}\subset X_{l-1}$ be the center of the fibered modification $\rho_l\colon X_l\to X_{l-1}$ where $l\in\{1,\ldots,m\}$. Then the subgroup $\mathbb{U}_{l}\subset\mathbb{U}_{l-1}$ acts identically on $C_{l-1}$.

\item[{\rm (e)}] If $\pi_X\colon X\to B$ has a reducible fiber then $\Aut_B(X_m)\cong \mathbb{U}_m\rtimes\mu_d$ where $\mu_d\subset\mathbb{G}_m$ is a finite cyclic group. 
\end{itemize}
\end{prop} 

\begin{proof} (a) By Corollary \ref{cor:jac} for any $l=0,\ldots,m$ one has a natural inclusion $\Aut_B(X_l)\subset\Aut_B(X_{0})$. Now, $\phi\in\Aut_B(X_{l-1})$ admits a lift to $\tilde\phi\in\Aut_B(X_{l})$ if and only if $\phi$ preserves the ideal of the center of the affine modification $\rho_l\colon X_l\to X_{l-1}$. This, clearly, leads to the inclusions $\Aut_B(X_l)\subset\Aut_B(X_{l-1})$. 

Statement
(b) is immediate from the fact that any $\phi\in\Aut_B(B\times\A^1)$ acts via
$$\phi\colon (b,u)\mapsto (b,\,\alpha(b)u+\beta(b))\quad\mbox{where}\quad \alpha\in\cO_B(B)^\times=\mathbb{k}^*\quad\mbox{and}\quad \beta\in\cO_B(B)\,.$$ 

The first equality in (c) is easy and is left to the reader. The second one follows from \cite[Thm.\ 6.3]{KPZ} (cf.\ also  \cite[Rem.\ 6.5.2]{KPZ}) due to the fact that locally in $B$ near the point $b_i$,
a locally nilpotent vertical vector field $\p$ on $B\times\A^1$ admits a lift to $X_l$ 
if and only if $\p$ is of the form $\p=f(z)z^{m_i,l}\p/\p u$ for a function $f$ in the local ring of $(B,b_i)$, cf.\ the proof of Lemma I.3.1. 

(d) The inclusion $\Aut_B(X_{l})\subset\Aut_B(X_{l-1})$ in \eqref{eq:nat-incl} implies that the union of the top level $l$ fiber components in $X_l$ is $\Aut_B(X_{l})$-invariant. Using \cite[Prop.\ 2.1]{KZ} one can conclude that
 $$\Aut_B(X_{l})=\Aut_B(X_{l-1},\,C_{l-1})\,.$$
 Since the ind-subgroup $\mathbb{U}_{l}\subset\Aut_B(X_{l})$ is connected and acts morphically on $X_{l-1}$, its action on the finite $\Aut_B(X_{l})$-invariant set
$C_{l-1}\subset X_{l-1}$ is identical. 

(e)
From (a)--(c) one obtains the inclusions 
\begin{equation}\label{eq:inclusions}
\Aut_B(X_l)/\mathbb{U}_l\subset\Aut_B(X_{l-1})/\mathbb{U}_{l-1}\subset\ldots\subset\Aut_B(X_{0})/\mathbb{U}_{0}=\mathbb{G}_m,\quad l=0,\ldots,m\,.\end{equation} 
Choose $l\in\{1,\ldots,m\}$ such that $\pi_{X_l}\colon X_l\to B$ has a reducible fiber while $\pi_{X_{l-1}}\colon X_{l-1}\to B$ does not. 
Then $\pi_{X_{l-1}}\colon X_{l-1}\to B$ admits a structure of a line bundle. Consider the associate $\mathbb{G}_m$-action on $X_{l-1}$ along the fibers of 
$\pi_{X_{l-1}}$. Due to \eqref{eq:inclusions} one has $\Aut_B(X_{l-1})=\mathbb{U}_{l-1}\rtimes\mathbb{G}_m$, and due to \eqref{eq:nat-incl},
$$\mathbb{U}_{l}\subset\Aut_B(X_{l})\subset\mathbb{U}_{l}\rtimes\mathbb{G}_m\,.$$
Let  $\pi_{X_{l}}^{-1}(b_i)$ be a reducible fiber. Then  the fiber $F=\pi_{X_{l-1}}(b_i)\cong\A^1$ contains at least two distinct points of $C_{l-1}$. Let $\phi$ be an element of the $\mathbb{G}_m$-action on $X_{l-1}$ along the fibers of $\pi_{X_{l-1}}$  which preserves the finite set $F\cap C_{l-1}$. Then $\phi$ has finite order, say, $d$ where $d\le {\rm card}\,(F\cap C_{l-1})$. Since $\phi^d$ is identical  on $F$ then also $\phi^d=\id$ on $X_{l-1}$. This shows that $\Aut_B(X_{l})=\mathbb{U}_{l}\rtimes\mu_d\,$ where $\mu_d\subset\mathbb{G}_m$ is a finite cyclic subgroup. A similar semi-direct product decomposition holds for any subgroup $\Aut_B(X_{i})\subset\Aut_B(X_{l})$, $i=l,\ldots,m$, see (c) and \eqref{eq:nat-incl}.
\end{proof}

\subsection{Configuration spaces and configuration invariants}\label{ss:config-space}

\bnota\label{nota:nat-aff-coord}
Consider  a marked GDF surface  $\pi_X\colon X\to B$ with a marking $z\in\cO_B(B)$, a trivializing sequence \eqref{eq: trivializing-seq-aff-modif}, and a graph divisor $\cD(\pi_X)=\sum_{i=1}^n \Gamma_ib_i$ where $z^*(0)=b_1+\ldots+b_n$. For a vertex $v\in{\rm vert}\,(\Gamma_i)$ on level $l>0$ we let $F(v)\subset X_l$ be the corresponding top level fiber component over $b_i$, and let $(b_i, u(v))\in B\times\A^1$ be the image  of $F(v)$ under the contraction $\rho_1\circ\ldots\circ\rho_l\colon X_l\to X_0=B\times\A^1$.  
Fixing the shortest path $\gamma(v)$ joining $v$ in $\Gamma_i$ with the root $v_{0,i}$ of $\Gamma_i$ one can construct the standard affine chart about $F(v)$ in $X_l$ following the blowup process along the path $\gamma(v)$. Starting with $t_{v_{i,0}}=u$, where $v_{i,0}$ is the root of $ \Gamma_i$, assume by recursion that $t_v$ is already defined for a (non-extremal) vertex $v$ of $ \Gamma_i$ on a level $l$, and let $w$ be a vertex on the level $l+1$ joint to $v$ by an edge. Then we let  \be\la{eq:coord} t_{w}=(t_v-t_v(w))/z\ee where $t_v(w)$ is the $t_v$-coordinate of the point $\rho_{l+1}(F(w))\in F(v)$.
 By recursion one defines  a system of affine coordinates on the fiber components of $z^{-1}(0)$ in $X_l$ for all $l=0,\ldots,m$.  If $X_l$ admits a $\mu_d$-action  along the fibers of $\pi_{X_l}\colon X_l\to B$ then this system of affine coordinates is $\mu_d$-invariant, cf. Lemma I.3.5 and its proof. 
\enota

\bdefi[\emph{Configuration space of a tree}]\label{def:config-tree}
Let $\Gamma$ be a rooted tree. For a vertex $v\in {\rm vert}\,(\Gamma)$ on level $l$ we let $r_+(v)$ be the number of  incident edges joining $v$ with vertices on level $l+1$. If $\h(\Gamma)>0$ then $r_+(v)=0$ if and only if $v$ is an extremal vertex different from the root $v_0$ of $\Gamma$.  Thus, $\sum_{v\in {\rm vert}\,(\Gamma)} r_+(v)$ equals the number of edges of $\Gamma$. We let  $\Gamma^*$ stand for the subgraph of $\Gamma$ obtained by deleting all the leaves of $\Gamma$ and their incident edges.

Let $D_r\subset\A^r$ be the zero level of the discriminant of the universal monic polynomial $p\in\mathbb{k}[t]$ of degree $r>0$. The open set $\mathcal{S}(r)=\A^r\setminus D_r\subset\A^r$ represents the configuration space of $r$-points subsets  (the roots of $p$) of $\A^1=\Spec\mathbb{k}[t]$. The \emph{configuration space of $\Gamma$} is the formal sum $$\mathcal{S}(\Gamma)=\sum_{v\in {\rm vert}\,(\Gamma^*)} \mathcal{S}(r_+(v)) v\,.$$ 
A point $\mathfrak{s}\in\mathcal{S}(\Gamma)$ can be viewed as a collection of configurations $$\mathfrak{s}=\{\mathfrak{s}(v)\in\mathcal{S}(r_+(v))\,|\,v\in {\rm vert}\,(\Gamma^*)\}\,.$$ The underlying variety of $\mathcal{S}(\Gamma)$ is the smooth affine variety $\prod_{v\in {\rm vert}\,(\Gamma^*)} \mathcal{S}(r_+(v))$ of dimension $$\dim \mathcal{S}(\Gamma)=\sum_{v\in {\rm vert}\,(\Gamma^*)} r_+(v)={\rm card}\,({\rm edges}\,(\Gamma))\,.$$ 
The automorphism group $\Aut(\Gamma)$ of the rooted tree $\Gamma$ acts on $\mathcal{S}(\Gamma)$ via
$$\Aut(\Gamma)\ni\alpha\colon \mathfrak{s}\mapsto \alpha_*(\mathfrak{s})\quad\mbox{where}\quad \alpha_*(\mathfrak{s})(v)=\mathfrak{s}(\alpha(v))\,.$$
This action is not effective, in general. Its kernel of non-effectiveness, say, $K$ is the pointwise stabilizer of $\Gamma^*$: $$K=\{\alpha\in\Aut(\Gamma)\,|\,
\alpha(v)=v\,\,\,\forall v\in {\rm vert}\,(\Gamma^*)\}\,.$$ The quotient $$\Aut^*(\Gamma)=\Aut(\Gamma)/K$$ acts effectively on $\mathcal{S}(\Gamma)$. 

Consider  the level $l$ factor $\mathcal{S}_l(\Gamma)$ of $\mathcal{S}(\Gamma)$ where
$$\mathcal{S}_l(\Gamma)=\sum_{v\in {\rm vert}\,(\Gamma^*)\,|\,l(v)=l} \mathcal{S}(r_+(v)) v\,.$$ 
There is an effective action of the  Abelian unipotent group $\mathbb{G}_a^h$ on $\mathcal{S}(\Gamma)$ where $h=\h(\Gamma)$. It is defined as follows. For $l=0,\ldots,m-1$ the $l$th $\mathbb{G}_a$-factor  of $\mathbb{G}_a^h$ acts identically on any component $\mathcal{S}(r_+(v))v$ with $l(v)\neq l$, and acts effectively on $\mathcal{S}_l(\Gamma)$
 via the  simultaneous shifts  on $\tau\in\mathbb{k}$ of the affine coordinates in all the corresponding configurations. This $\mathbb{G}_a$-action on  $\mathcal{S}_l(\Gamma)$ is free and admits a slice $$\mathcal{S}^{\mathrm o}_l(\Gamma)=\left\{\mathfrak{s_l}=\sum_{v\in {\rm vert}\,(\Gamma^*)\,|\,l(v)=l} \mathfrak{s_l}(v)v\,|\,\sum_{v\in {\rm vert}\,(\Gamma^*)\,|\,l(v)=l} {\rm barycentre}\,(\mathfrak{s_l}(v))=0\right\}\,.$$ The smooth $\Aut^*(\Gamma)$-invariant affine variety
 $\mathcal{S}^{\mathrm o}(\Gamma):=\prod_{l=0}^{m-1} \mathcal{S}^{\mathrm o}_l(\Gamma)$ 
 is a slice of the free $\mathbb{G}_a^h$-action on $\mathcal{S}(\Gamma)$.

 Besides, there is an effective $\G_m$-action on  $\mathcal{S}(\Gamma)$  via the  simultaneous multiplication by $\lambda\in\mathbb{k}^*$ of the affine coordinates in all the corresponding configurations. This action leaves the slice $\mathcal{S}^{\mathrm o}(\Gamma)$ invariant.
 
  The action of the semi-direct product $$G(\Gamma)=\mathbb{G}_a^h\rtimes (\G_m\times\Aut^*(\Gamma))$$ on $\mathcal{S}(\Gamma)$ descends to an effective action of the reductive group $\G_m\times\Aut^*(\Gamma)$ on the smooth affine variety $\mathcal{S}(\Gamma)/\mathbb{G}_a^h\cong \mathcal{S}^{\mathrm o}(\Gamma)$. The quotient 
  $$\mathfrak{M}(\Gamma)=\mathcal{S}(\Gamma)/G(\Gamma)=\mathcal{S}^{\mathrm o}(\Gamma)/(\G_m\times\Aut^*(\Gamma))$$ is an affine variety with quotient singularities. 
\edefi

\bdefi[\emph{Configuration space of a graph divisor}]\label{def:config-graph}
Given a graph divisor $\cD=\sum_{i=1}^n \Gamma_ib_i$ we define the \emph{configuration space of $\cD$} to be the formal sum $$\mathcal{S}(\cD)=\sum_{i=1}^n  \mathcal{S}(\Gamma_i) b_i\,.$$ The underlying smooth affine variety $\prod_{i=1}^n\mathcal{S}(\Gamma_i)$ has dimension equal to the number of edges in $\cD$. 

Letting $\cD^*=\sum_{i=1}^n \Gamma_i^*b_i$ consider the group $\Aut^*_B(\cD)=\prod_{i=1}^n\Aut^*(\Gamma_i)$. It acts effectively on  $\mathcal{S}(\cD)$.  

Let $h_i=\h(\Gamma_i)$, and let $h(\cD)=\sum_{i=1}^n h_i$. The effective action of $\mathbb{G}_a^{h_i}\rtimes\G_m$ on $\mathcal{S}(\Gamma_i)$ as defined in \ref{def:config-tree} extends to an effective action of $\mathbb{G}_a^{h(\cD)}\rtimes\G_m$ on $\mathcal{S}(\cD)$. The action of $\mathbb{G}_a^{h(\cD)}$ is free and admits a slice
$$\mathcal{S}^{\mathrm o}(\cD)=\sum_{i=1}^n \mathcal{S}^{\mathrm o}(\Gamma_i) b_i\,.$$
The action of the affine algebraic group 
\be\la{eq:G(D)} G(\cD)=\mathbb{G}_a^{h(\cD)}\rtimes (\G_m\times\Aut^*_B(\cD))\ee 
on $\mathcal{S}(\cD)$ descends to an effective action of $\G_m\times\Aut^*_B(\cD)$ on the quotient $\mathcal{S}(\Gamma)/\mathbb{G}_a^{h(\cD)}\cong \mathcal{S}^{\mathrm o}(\cD)$. The quotient
\be\la{eq:moduli} 
\mathfrak{M}(\cD)=\mathcal{S}(\cD)/G(\cD)=\mathcal{S}^{\mathrm o}(\cD)/(\G_m\times\Aut_B^*(\cD))\ee is an affine variety with quotient singularities. 
\edefi 

\bdefi[\emph{Configuration invariant}]\label{def:config-invariant} 
Let $\pi_X\colon X\to B$  be a marked GDF surface  as in \ref{nota:nat-aff-coord} with the graph divisor $\cD=\cD(\pi_X)=\sum_{i=1}^n \Gamma_ib_i$ where $\Gamma_i=\Gamma_{b_i}(\pi_X)$. For a vertex $v\in {\rm vert}\,(\cD^*)$ the top level $l$ fiber component $F(v)\subset X_l$ is equipped with the natural affine coordinate $t_v$ as in \ref{nota:nat-aff-coord}. Let $C_l\subset X_l$ be the center of the fibered modification $\rho_{l+1}\colon X_{l+1}\to X_l$ in \eqref{eq: trivializing-seq-aff-modif}. We define the
\emph{configuration invariant} $\Delta(\pi_X)\in\mathcal{S}(\cD(\pi_X))$ by the formula
$$\Delta(\pi_X)(v)=F(v)\cap C_l\in \mathcal{S}(r_+(v)),\quad v\in {\rm vert}\,(\cD^*), \,\,\,l(v)=l\,.$$ 
If $X$ admits a $\mu_d$-action  along the fibers of $\pi_{X}$ then this action is well defined on $X_l$ for any $l=0,\ldots,m$ making  \eqref{eq: trivializing-seq-aff-modif} equivariant. So, both the collection of centers $(C_l)_{l=0,\ldots,m-1}$ and the collection of the affine coordinates $t_v$ are invariant under $\mu_d$. Hence the configuration invariant $\Delta(\pi_X)$ is as well. 
\edefi

\subsection{Versal deformation families of trivializing sequences}\la{ss:versal}

\begin{defi}\label{def:deform-family} Consider a marked GDF surface $\pi_X\colon X\to B$ with a marking $z\in\cO_B(B)\setminus\{0\}$ and a trivializing sequence \eqref{eq: trivializing-seq-aff-modif}. 
By a corresponding \emph{family of trivializing sequences} we mean a commutative diagram of varieties and morphisms
\be\la{diagr:versal}
\begin{diagram}[notextflow]
   \mathfrak{X}=\mathfrak{X}_m\,\,\,\,\,\,\,\,\,\,       & \rTo^{r_m}       &   \mathfrak{X}_{m-1}  &  \rTo<{}>{} & \ldots & \rTo<{}>{} & \mathfrak{X}_{1} & \rTo<{r_1}>{} & \mathfrak{X}_{0}  \\
    & \rdTo_{}      &        & \rdTo_{}  &  & \ldots &  & \rdTo_{} &  &  \rdTo_{} \\
\dTo^{}   &    & B \,\,\,    \dTo  &\rTo^{\id}  &  B & \rTo<{}>{} \ldots& \dTo^{}  &\rTo<{}>{} &    B   \dTo^{} & \rTo<{\id}>{} & B\\
  \mathcal{S}\   &   \rTo^{\id} &   \mathcal{S}         & \rTo<{}>{} & \ldots &\rTo<{}>{} & \mathcal{S} & \rTo<{\id}>{} & \mathcal{S}\\
\end{diagram}
\ee
where 
\begin{itemize}
\item $\mathfrak{X}_0=\mathcal{S}\times (B\times\A^1)$ with the standard projections to the factors $\mathcal{S}$ and $B$;
\item $\mathfrak{X}_l\to \mathcal{S}$ is a family of GDF surfaces for any $l=0,\ldots,m$;
\item there is a point $\mathfrak{s}_0\in \mathcal{S}$ such that the specialization of the upper line in \eqref{diagr:versal} over $\mathfrak{s}_0$ yields the initial trivializing sequence $(\pi_{X_l}\colon X_l\to B)_{l=0,\ldots,m}$.   
\end{itemize}

We say that \eqref{diagr:versal} is an \emph{affine deformation family of trivializing sequences} if, for any $l=0,\ldots,m$, $\mathfrak{X}_l\to \mathcal{S}$ is an affine deformation family of marked GDF surfaces. It is easily seen that \eqref{diagr:trivial} extends to a diagram of trivializing families of SNC completions which specializes to \eqref{eq: trivializing-seq-blowups} in each fiber, along with a  locally trivial family of extended divisors $(D_{\rm ext}(\mathfrak{s}))_{\mathfrak{s}\in\mathcal{S}}$. 

An affine deformation family \eqref{diagr:versal} is called \emph{trivial} if for any $l=1,\ldots,m$ there is a commutative diagram
\be\label{diagr:trivial}
\begin{diagram}[notextflow]
   \mathcal{S}\times X_l &    &\rTo^{\id\times\rho_l} &      &   \mathcal{S}\times X_{l-1}   \\
      & \rdTo_{\cong} &      &      & \vLine^{}& \rdTo_{\cong}  \\
\dTo^{} &    &   \mathfrak{X}_l   & \rTo^{r_l}  & \HonV   &    &  \mathfrak{X}_{l-1}  \\
      &    & \dTo^{}  &      & \dTo   \\
   \mathcal{S}  & \hLine & \VonH   & \rTo^{\rm id} &    \mathcal{S}   &    & \dTo_{} \\
      & \rdTo_{\id} &      &      &      & \rdTo_{\id}  \\
      &    &    \mathcal{S}   &      & \rTo^{\rm id}  &    &   \mathcal{S}  \\
\end{diagram}
\ee
where the morphisms in the upper square are defined over $B$.
\end{defi}

\begin{sit}[\emph{Deformation families of GDF surfaces over configuration spaces}]\label{lem:versal-fam-config} Given a triplet $(B,z,\cD)$ as in Definition \ref{def:coarse-moduli} there is a natural deformation family $$\mathfrak{F}(\cD)=(\mathfrak{X}_l(\cD)\to\mathcal{S}(\cD))_{l=0,\ldots,m} $$ of trivializing sequences
 \eqref{diagr:versal} of GDF surfaces of type $(B,z,\cD)$ over the configuration space 
$\mathcal{S}(\cD)$ 
 such that the associated family of graph divisors  is trivial. 
The construction of the latter family proceeds as follows.

We let $\mathfrak{X}_0(\cD)=\mathcal{S}(\cD)\times (B\times\A^1)$ with the canonical projections to the first and the second factors. Projecting $\mathcal{S}(\cD)$ to the zero level factor $$\mathcal{S}_0(\cD)=\sum_{i=1}^n \mathcal{S}(v_{0,i})v_{0,i}b_i$$ where $v_{0,i}$ is the root of $\Gamma_i$ yields a multisection of the first projection $\mathfrak{X}_0(\cD)\to\mathcal{S}(\cD)$. This multisection defines the center $ \mathfrak{C}_0(\cD)$ of the fibered modification $r_1(\cD)\colon \mathfrak{X}_1(\cD)\to\mathfrak{X}_0(\cD)$ fitting in diagram  \eqref{diagr:versal}. Projecting now to the 
first level factor  $$\mathcal{S}_1(\cD)=\sum_{i=1}^n\left( \sum_{v\in{\rm vert}\,(\Gamma^*_i)\,|\,l(v)=1}\mathcal{S}(v)v\right) b_i$$ yields the center  $\mathfrak{C}_1(\cD)$ of the fibered modification $r_2(\cD)\colon \mathfrak{X}_2(\cD)\to\mathfrak{X}_1(\cD)$ fitting in \eqref{diagr:versal}, etc. Continuing in this way one arrives finally at a deformation family $\mathfrak{F}(\cD)$ over $\mathcal{S}=\mathcal{S}(\cD)$ of trivializing sequences of marked GDF surfaces $(X_l(\mathfrak{s}))_{\mathfrak{s}\in\mathcal{S}, l=0,\ldots,m}$  fitting in \eqref{diagr:versal} with the marking $z$, the trivial family of graph divisors $\cD(\pi_{X(\mathfrak{s})})=\cD$, and the configuration invariant $\Delta(X(\mathfrak{s}))=\mathfrak{s}\in\mathcal{S}(\cD)$ where $X(\mathfrak{s})=X_m(\mathfrak{s})$.
Clearly, it admits an extension to a family of complete surfaces with a trivial family of extended divisors.
\end{sit}

Any morphism $f\colon\mathcal{S}\to \mathcal{S}(\cD)$ induces a family $\mathfrak{F}=f^*(\mathfrak{F}(\cD))$ of trivializing sequences over $\mathcal{S}$ with a trivial family of graph divisors. Conversely, any such family arises in this way, which shows the versality of $\mathfrak{F}(\cD)$. Namely, the following holds. 

\begin{prop}\label{lem:versality-config} 
Let $\mathfrak{F}=(\mathfrak{X}_l\to\mathcal{S})_{l=0,\ldots,m}$ be a family of trivializing sequences over the same base $\mathcal{S}$. Assume that the associated family of graph divisors over $\mathcal{S}$ is trivial:
$$\cD(\pi_{X(\mathfrak{s})})=\cD\quad\forall \mathfrak{s}\in \mathcal{S}\,.$$ Then one has $\mathfrak{F}=\Delta^*(\mathfrak{F}(\cD))$, that is, $\mathfrak{F}$ is induced from $\mathfrak{F}(\cD)$ via the morphism 
$$\Delta\colon \mathcal{S}\to \mathcal{S}(\cD),\quad \mathfrak{s}\mapsto \Delta(X(\mathfrak{s}))\,$$ defined by the configuration invariant.  Consequently, the deformation family of trivializing sequences $\mathfrak{F}(\cD)$ is  versal with respect to the \'etale topology.
\end{prop}

\begin{proof}
We proceed by recursion on $m$. The assertion is evidently true if $m=0$. Suppose it holds for $m=l-1$, that is, the lower square in the following diagram  is commutative:
\be\label{diagr:induced}
 \bdi
 \mathfrak{X}_l& \rTo>{\phi_l} &  \mathfrak{X}_l(\cD)\\
\dTo<{r_l} &  & \dTo>{r_l(\cD)} \\
\mathfrak{X}_{l-1}&\rTo<{}>{\phi_{l-1}} &  \mathfrak{X}_{l-1}(\cD)\\
\dTo<{} &  & \dTo>{} \\
\mathcal{S}&\rTo^{\Delta}&  \mathcal{S}(\cD)\, 
\edi
\ee 
For $l=0,\ldots,m-1$ we let $\Delta_l\colon \mathcal{S}\to \mathcal{S}_l(\cD)$ be the composition of $\Delta$ with the projection to the level $l$ component $\mathcal{S}_l(\cD)$ of $ \mathcal{S}(\cD)$. The image of $\Delta_{l-1}$ can be seen as a multisection of $\mathfrak{X}_{l-1}\to\mathcal{S}$ which defines the center $\mathfrak{C}_{l-1}\subset\mathfrak{X}_{l-1}$ of the fibered modification 
$r_l\colon\mathfrak{X}_l\to\mathfrak{X}_{l-1}$. Due to diagram \eqref{diagr:induced} one has
$\mathfrak{C}_{l-1}=\phi_{l-1}^*(\mathfrak{C}_{l-1}(\cD))$. According to \cite[Prop.\ 2.1]{KZ}, $\phi_{l-1}$ admits a lift to a morphism $\phi_l\colon  \mathfrak{X}_l\to \mathfrak{X}_l(\cD)$ which makes the upper square in 
\eqref{diagr:induced} commutative. This gives the recursive step and proves the first assertion. 

To show the second one it suffices to notice that any family of graph divisors has a finite monodromy group. This monodromy becomes trivial after a suitable \'etale base change. According to the first part, the resulting family is induced from $\mathfrak{F}(\cD)$ via the morphism given by the configuration invariant. 
\end{proof}

Let us study the automorphism group of $\mathfrak{F}(\cD)$ over $B$.

\begin{lem}\label{lem:isomorphism} Consider two fibers $X_l=X_l(\mathfrak{s})$ and $X'_l=X_l(\mathfrak{s'})$ of $\mathfrak{F}_l(\cD)$ where $\mathfrak{s},\,\mathfrak{s'}\in\mathcal{S}(\cD)$, $l=0,\ldots,m$. Assume that there is an isomorphism $\phi\colon X=X_m\stackrel{\cong_B}{\longrightarrow} X'=X'_m$ which induces the identity on $\cD$. Then $\phi$ induces for any $i=1,\ldots,n$ and $l=0,\ldots,h_i-1=\h(\Gamma_i^*)$ an affine transformation
\be\la{eq:affine} \phi_{i,l}\colon t_v\mapsto \alpha t_v + \beta_{i,l},\quad \alpha_i\in\mathbb{k}^*,\,\,\,\beta_{i,l}\in\mathbb{k}
\ee
such that $\mathfrak{s'}(v)=\phi_{i,l}(\mathfrak{s}(v))$ for any $v\in {\rm vert}\,(\Gamma^*_i)$ with $l(v)=l$.
\elem

\bproof
Recall that $\phi$ can be extended first to an isomorphism  of pseudominimal completions 
$\bar\phi\colon\bar X\stackrel{\cong_{\bar B}}{\longrightarrow}\bar X'$, and then to  an isomorphism of the trivializing sequences  of completions \eqref{eq: trivializing-seq-blowups} which yields the identity on $\Gamma_i$ for any $i=1,\ldots,n$, see  the proof of Proposition I.8.3(c). By Corollary \ref{cor:jac} the associated birational transfomation $\psi$ of $B\times\A^1$ over $B$ fitting in \eqref{diagr: star} is biregular. Due to our assumption $\cO_B(B)^\times =\mathbb{k}^*$. Hence $\psi$ is of the form 
\be\la{eq:shear} \psi\colon (b,u)\mapsto (b,\alpha u+\beta)\quad\mbox{where}\quad \alpha\in \mathbb{k}^* \quad\mbox{and}\quad\beta\in\cO_B(B)\,.\ee 
In particular,  $\psi_*(z)=z$. Using $z$ as a local parameter in $B$ near $b_i$ one can write
\be\la{eq:develop} \beta(z)=\beta_{i,0}+\beta_{i,1} z +\ldots +\beta_{i,l}z^l+\ldots \quad\mbox{with}\quad \beta_{i,j}\in\mathbb{k}\,\,\,\forall i,j\,.\ee 
Consider the standard local charts $U(v)\subset X_l$ about $F(v)$ with local coordinates 
$(z,t_v)$ and  $U(w)\subset X_{l+1}$ about $F(w)$  with local coordinates $(z,t_w)$ and, respectively, $U'(v)\subset X'_l$ about $F'(v)$ with local coordinates 
$(z,t'_v)$ and  $U'(w)\subset X'_{l+1}$ about $F'(w)$  with local coordinates $(z,t'_w)$.
We claim that the restriction $\phi|_{U(v)}\colon U(v)\mapsto U'(v)$ is given by
\be\la{eq:loc-charts} (z,t'_v)=\psi_*(z,t_v)=(z,\alpha t_v + \beta_{i,l}+\beta_{i,l+1}z+\ldots)\,.\ee
Indeed,  for $l=0$, \eqref{eq:loc-charts} follows from \eqref{eq:shear} and \eqref{eq:develop}. 
Suppose by induction that \eqref{eq:loc-charts} holds for a given $l\le h_i-2$. Let $v, w\in {\rm vert}\,(\Gamma^*_i)$ with $l(v)=l$ and $l(w)=l+1$ be joint by an edge, and let $c(w)=\rho_{l+1}(F(w))\in F(v)$.
 Applying the inductive hypothesis one obtains 
\be\la{eq:base-ind} \phi_*\colon (z,t_v)\mapsto (z,\alpha t_v+\beta_{i,l}+\beta_{i,l+1}z+\beta_{i,l+2}z^2+\ldots),\,\,\,c(w)\mapsto\alpha c(w)+\beta_{i,l}\,.\ee  Using \eqref{eq:coord} and \eqref{eq:base-ind} one gets
$$\phi(z,t_w)=(z,\,(t_v-c(w))/z)\stackrel{\phi_*}{\longmapsto} (z,\,\alpha(t_v-c(w))/z+\beta_{i,l+1}+\beta_{i,l+2}z+\ldots)=(z,\,\alpha t_w+\beta_{i,l+1}+\beta_{i,l+2}z+\ldots)\,.$$ 
This gives the inductive step. Letting $z=0$ in \eqref{eq:base-ind} yields \eqref{eq:affine}.
\eproof

\begin{prop}\label{lem:univ-grp-action} \begin{itemize}\item[{\rm (a)}]
The action of the group $G(\cD)$ on $\mathcal{S}(\cD)$ defined in {\rm \ref{def:config-graph}}  admits a lift to an action of $G(\cD)$ on $\mathfrak{X}_l(\cD)$ over $B$ making  the following morphisms in \eqref{diagr:versal} $G(\cD)$-equivariant:
\be\la{eq:morphisms} \mathfrak{X}_l(\cD)\to\mathcal{S}(\cD), \,\,\,l=0,\ldots,m\quad\mbox{and}\quad r_{l}\colon\mathfrak{X}_{l}(\cD)\to\mathfrak{X}_{l-1}(\cD), \,\,\,l=1,\ldots,m\,.\ee 

\item[{\rm (b)}] Letting $\mathfrak{X}(\cD)=\mathfrak{X}_m(\cD)$ one has $\Aut_B(\mathfrak{X}(\cD)\to \mathcal{S}(\cD))=G(\cD)$. 

\item[{\rm (c)}] For $\mathfrak{s},\mathfrak{s}'\in \mathcal{S}(\cD)$ the surfaces $X(\mathfrak{s})$ and $X(\mathfrak{s}')$  are isomorphic over $B$ if an only if $\mathfrak{s}$ and $\mathfrak{s}'$ lie on the same $G(\cD)$-orbit.

\item[{\rm (d)}] The retraction  $\eta\colon \mathcal{S}(\cD)\to \mathcal{S}^{\mathrm o}(\cD)\cong \mathcal{S}(\cD)/\G_m^{h(\cD)}$, see Definition {\rm\ref{def:config-graph}}, defines a trivial $\G_m^{h(\cD)}$-bundle over $\mathcal{S}^{\mathrm o}(\cD)$. 

\item[{\rm (e)}] The family $\mathfrak{F}(\cD)=(\mathfrak{X}(\cD)\to\mathcal{S}(\cD))$  is isomorphic to the family induced from the restriction $\mathfrak{F}^{\mathrm o}(\cD)=(\mathfrak{X}^{\mathrm o}(\cD)\to\mathcal{S}^{\mathrm o}(\cD))$ via the retraction morphism $\eta$. Consequently,  the deformation family $\mathfrak{F}^{\mathrm o}(\cD)$ of GDF surfaces of type $(B,z,\cD)$ is versal  with respect to the \'etale topology.
\end{itemize}
\eprop

\bproof (a) We proceed by induction on $l$. The assertion is trivially true for $l=0$. Suppose it holds for some $l\in\{0,\ldots,m-1\}$. Then $g\in G(\cD)$ defines for any $\mathfrak{s}\in\mathcal{S}(\cD)$ an isomorphism $$g|_{X_l(\mathfrak{s})}\colon X_l(\mathfrak{s})\stackrel{\cong_B}{\longrightarrow} X_l(g(\mathfrak{s}))\,$$
satisfying \eqref{eq:affine}, see Lemma \ref{lem:isomorphism}.
The component $\mathfrak{s}_l\in \mathcal{S}_l(\cD)$ of $\mathfrak{s}$ is sent to the 
 component $g(\mathfrak{s})_l\in \mathcal{S}_l(\cD)$ of $g(\mathfrak{s})$. It follows that the center $\mathfrak{C}_l(\cD)\subset\mathfrak{X}_{l}(\cD)$ of the affine modification $r_l(\cD)\colon \mathfrak{X}_{l+1}(\cD)\to\mathfrak{X}_{l}(\cD)$ is $G(\cD)$-invariant (cf.\ \ref{lem:versal-fam-config}). Since its divisor $z^*(0)$ is $G(\cD)$-invariant too, due to Lemma I.1.5 the action of $G(\cD)$ on $\mathfrak{X}_{l}(\cD)$ admits a lift to $\mathfrak{X}_{l+1}(\cD)$ making the morphisms \eqref{eq:morphisms} $G(\cD)$-equivariant. 
 
 (b) By (a) one has $\Aut_B(\mathfrak{X}(\cD)\to \mathcal{S}(\cD))\supset G(\cD)$. Applying the same argument to $g\in \Aut_B(\mathfrak{X}(\cD)\to \mathcal{S}(\cD))$ one concludes that $g$
satisfies \eqref{eq:affine}, hence belongs to $G(\cD)$. 

(c) By virtue of (a),  if $\mathfrak{s}$ and $\mathfrak{s}'$ lie on the same $G(\cD)$-orbit then $X(\mathfrak{s})\cong_B X(\mathfrak{s}')$. Suppose further that $X(\mathfrak{s})\cong_B X(\mathfrak{s}')$. Composing this isomorphism, say, $\phi$ with a suitable $\alpha\in\Aut_B(\cD)$ acting on $\mathfrak{X}(\cD)$ one may assume that $\phi$ induces the  identity on $\cD$. Then by Lemma \ref{lem:isomorphism}, $\phi$ satisfies \eqref{eq:affine}, and so, extends to an element of $G(\cD)$ acting on $\mathfrak{X}(\cD)$. Since $\mathfrak{X}(\cD)\to \mathcal{S}(\cD)$ is $G(\cD)$-equivariant it follows that $\mathfrak{s}$ and $\mathfrak{s}'$ lie on the same $G(\cD)$-orbit.
\eproof

(d) The $\mathbb{G}_a^{h(\cD)}$-equivariant isomorphism 
$$ \Phi\colon \mathcal{S}^{\mathrm o}(\cD)\times \mathbb{G}_a^{h(\cD)}\stackrel{\cong}{\longrightarrow} \mathcal{S}(\cD),\quad (\mathfrak{s}, g)\mapsto g(\mathfrak{s})$$
 yields the desired equivariant trivialization. 
 
 (e) Let 
 $\eta^*(\mathfrak{F}^{\mathrm o}(\cD))=
 (\mathfrak{X}'(\cD))\to \mathcal{S}(\cD))$ be the induced family. Since $\mathcal{S}^{\mathrm o}(\cD)$ is a slice for the $\mathbb{G}_a^{h(\cD)}$-action on $\mathcal{S}(\cD)$ and the projection $\pi_{\mathfrak{X}(\cD)}\colon \mathfrak{X}(\cD)\to  \mathcal{S}(\cD)$ is $\mathbb{G}_a^{h(\cD)}$-equivariant, see (a), then also $\mathfrak{X}^{\mathrm o}(\cD)$ is a slice for the free $\mathbb{G}_a^{h(\cD)}$-action on $\mathfrak{X}(\cD)$, cf.\ (b). 
Therefore, one has a commutative diagram of $\mathbb{G}_a^{h(\cD)}$-equivariant morphisms 
\be
 \bdi\la{diagr:equivar}
\mathfrak{X}(\cD)^{\mathrm o}\times \mathbb{G}_a^{h(\cD)} &\rTo<{\widetilde\Phi}>{\cong} & \mathfrak{X}(\cD)\\
\dTo<{} &  & \dTo>{} \\
 \mathcal{S}^{\mathrm o}(\cD)\times \mathbb{G}_a^{h(\cD)}&\rTo>{\Phi}<{\cong} &  \mathcal{S}(\cD)\, 
\edi
\ee
where $\widetilde\Phi\colon (x, g)\mapsto g(x)$. Using Proposition \ref{lem:versality-config} and the latter diagram the assertions follow.

\brems\la{rem:interpret}
1.
The action of the unipotent radical $R_u(G(\cD))=\mathbb{G}_a^{h(\cD)}$ on 
$\mathfrak{X}(\cD)$ has the following interpretation. The group $\Aut_B(B\times\A^1)=\mathbb{U}_0\rtimes\mathbb{G}_m$ where $\mathbb{U}_0=\cO_B(B)$ acts naturally on $\mathcal{S}(\cD)$ via \eqref{eq:shear}. For $\mathfrak{s}\in\mathcal{S}(\cD)$ the subgroup $\mathbb{U}_{X(\mathfrak{s})}=\mathbb{U}_m$ from Proposition \ref{prop:aut-GDF} is the stabilizer of $\mathfrak{s}$ in $\mathbb{U}$. By Proposition \ref{prop:aut-GDF}(c) for any $\mathfrak{s}'\in \mathcal{S}(\cD)$ one has 
$$\mathbb{U}_{X(\mathfrak{s}')}=\mathbb{U}_{X(\mathfrak{s})}\cong H^0(B,-D_m)\quad\mbox{where}\quad D_m=\sum_{i=1}^n h_{i}b_i\,.$$  By Proposition \ref{prop:aut-GDF}(d), $\mathbb{U}_X=\mathbb{U}_{X(\mathfrak{s})}$ acts identically on $\mathcal{S}(\cD)$. The quotient ${\mathbb{U}}_0/{\mathbb{U}}_X\cong \mathbb{G}_a^{h(\cD)}$ acts transitively on each orbit of ${\mathbb{U}}_0$ in $\mathcal{S}(\cD)$, and this action coincides with the $\mathbb{G}_a^{h(\cD)}$-action.

2. It can be shown  that the quotient $\mathfrak{M}(\cD)=\mathcal{S}(\cD)/G(\cD)$ does not depend on the choice of a trivializing sequence for $\pi_X\colon X\to B$. Anyway, this fact follows also from the proof of Theorem \ref{thm:moduli}(a) below. 
\erems

\subsection{Proof of Theorem \ref{thm:moduli}}
(a) 
We claim that the desired coarse moduli space is $$\mathcal{C}(B,z, \cD)=\mathfrak{M}(\cD)=\mathcal{S}(\cD)/G(\cD)=\mathcal{S}^{\mathrm{o}}(\cD)/(\G_m\times\Aut_B^*(\cD))\,,$$ see \eqref{eq:moduli}. Indeed, consider a deformation family $\mathfrak{F}: (\eta\colon\mathfrak{X}\to\mathcal{S},\, \,\pi\colon\mathfrak{X}\to B)$   of marked GDF surfaces $\pi|_{X(\mathfrak{s})}\colon X(\mathfrak{s})=\eta^{-1}(\mathfrak{s})\to B$ sharing the common marking $z\in\cO_B(B)$ and the common graph divisor $\cD(\pi|_{X(\mathfrak{s})})\cong_B\cD$ where the latter isomorphism depends on $\mathfrak{s}\in\mathcal{S}$. 

 Passing to the Galois covering $\mathcal{S}'\to\mathcal{S}$ defined by the monodromy group of the latter family we obtain the induced family $\mathfrak{F}': (\mathfrak{X}'\to\mathcal{S}')$ with a trivial family of graph divisors. By Proposition \ref{lem:versality-config} the configuration invariant defines a morphism 
$\Delta\colon\mathcal{S}'\to\mathcal{S}(\cD)$ such that $\mathfrak{F}'=\Delta^*(\mathfrak{F}(\cD))$. The quotient morphism $\mathcal{S}'\to\mathcal{S}(\cD)/G(\cD)$ is constant on any fiber of $\mathcal{S}'\to\mathcal{S}$. Hence $\Delta$ can be factorized via 
a morphism $\delta\colon\mathcal{S}\to\mathcal{S}(\cD)/G(\cD)=\mathfrak{M}(\cD)$. Due to Proposition \ref{lem:univ-grp-action}(c) such a morphism $\delta$ is uniquely defined and has the desired properties, see Definition \ref{def:coarse-moduli}. 

Statement (b) follows from Theorem \ref{thm: main0}. 

Statement (c) is immediate by Proposition I.7.15.
Indeed, by virtue of this proposition, performing a top level stratching one replaces $\cD$ by $\cD^{(k)}$ without changing 
the isomorphism class over $B$ of the cylinder $X\times\A^1$. Since $X$ has a reducible fiber the graph divisor $\cD^{(k)}$ of the resulting marked GDF surface $X^{(k)}\to B$ with the same marking $z$ has at least $2k$ additional edges. So, one has $$\dim\mathfrak{M}(\cD)\ge {\rm card}\,({\rm edges}\,(\cD))-h(\cD)-1\quad\mbox{and}\quad 
\dim\mathfrak{M}(\cD^{(k)})\ge \dim\mathfrak{M}(\cD) +k\,,$$ which implies the assertion.
\qed

\bcor\la{cor:chain-div}
Any deformation family of GDF surfaces whose graph divisors are chain divisors is trivial. 
\ecor

\bproof Let a GDF surface $\pi_X\colon X\to B$ admits a line bundle structure, that is, $\cD(\pi_X)=\cD$ is a chain divisor. Then ${\rm card}\,({\rm edges}\,(\cD))=h(\cD)$, $\mathcal{S}(\cD)\cong\A^{h(\cD)}$,  and the $\G_a^{h(\cD)}$-action on $\mathcal{S}(\cD)$ is simply transitive. Hence $\mathcal{S}^{\mathrm o}(\cD)$  is a singleton, and so, the versal deformation family $\mathfrak{X}(\cD)\to\mathcal{S}(\cD)$ is trivial by Proposition \ref{lem:univ-grp-action}(d). Furthermore, any locally trivial family of chain divisors is trivial since the group $\Aut(\cD)$ is. 
Since any deformation family $\mathfrak{F}: \mathfrak{X}\to\mathcal{S}$ with the given graph divisor $\cD$ is induced from $\mathfrak{X}(\cD)\to\mathcal{S}(\cD)$ via the  morphism $\Delta:\mathcal{S}\to\mathcal{S}(\cD)$  given by the configuration invariant, see Proposition \ref{lem:versality-config}(e), $\mathfrak{F}$ is trivial as well. 
\eproof

\brem\la{rem:more-general} In fact, this corollary holds without any assumption on $B$. 
Theorem \ref{thm:moduli} remains valid if one replaces the assumption ``$\cO_B(B)^\times = \mathbb{k}^*$'' by the following one: ``$z^{-1}(0)$ is a singleton''. 
However, in general the coarse moduli space $\mathcal{C}_1(B,z, \cD)$ does not exist  in any reasonable category of spaces. Let us give a simple example.
\erem

\bexa\la{ex:non-algebraic} Let $\mathbb{k}=\mathbb{C}$, and let $B=\A^1\setminus\{0,\pi\}$, where $\A^1=\Spec \mathbb{C}[t]$, be equipped with the marking $z=t^2-1$ so that $n=2$ and $b_1=1,\,b_2=-1$.
Let also $\Gamma_i$,  $i=1,2$, be the rooted tree of height $1$ with two vertices on level $1$. One has $$\mathcal{S}(\Gamma_i)=\mathcal{S}(2),\quad\mathcal{S}^{\mathrm o}(\Gamma_i)\cong\A^1_*:=\A^1\setminus\{0\},\,\,  i=1,2,\quad\mbox{and}\quad\mathcal{S}^{\mathrm o}(\cD)\cong (\A^1_*)^2\,.$$ The group $\Aut^*(\cD)$ is trivial, see Definition \ref{def:config-graph}. The infinite discrete group $$\cO_B(B)^\times/\mathbb{k}^*\cong\ZZ^2\,$$ acts  naturally on the quotient $$\mathcal{S}^{\mathrm o}(\cD)/(\G_m\times\Aut^*(\cD))\cong (\A^1_*)^2/\G_m\cong\A^1_*$$ as a subgroup of $\G_m\subset\Aut(\A^1_*)$ generated by two nonzero complex numbers whose ratio is transcendental. Its orbits correspond to the isomorphism classes of the associated GDF surfaces.
It is easily seen that the induced complex topology of
the quotient $\A^1_*/\ZZ^2$ does not satisfy the Kolmogorov $T_0$ axiom. Hence this quotient neither is an algebraic (or complex) space, nor is an algebraic  stack. 
\eexa

\end{document}